\documentclass[11pt,fleqn]{amsart} 

\usepackage[usenames,dvipsnames,condensed]{xcolor}

\usepackage{amsthm}
\usepackage{amsfonts}
\usepackage[english]{babel}
\usepackage[usenames]{xcolor}
\usepackage{graphicx}
\usepackage{soul}
\usepackage{stfloats}
\usepackage{morefloats}
\usepackage{cite}
\usepackage{lscape}
\usepackage{epstopdf}
\usepackage{bbm}
\usepackage[lite]{amsrefs}
\usepackage{amsmath,amstext,amsopn,amsfonts,eucal,amssymb}
\usepackage{graphicx,wrapfig,url}

\newcommand\N{{\mathbb N}}

\newcommand\Q{{\mathbb Q}}

\newtheorem{theorem}{Theorem}[section]
\newtheorem{corollary}[theorem]{Corollary}
\newtheorem{lemma}[theorem]{Lemma}
\newtheorem{proposition}[theorem]{Proposition}

\newtheorem{remark}[theorem]{Remark}

\newtheorem{conjecture}[theorem]{Conjecture}

\newcommand{\sym}{e_1\otimes e_2 - e_2\otimes e_1}
\newcommand{\lwall}{e_2-e_1}
\newcommand{\rwall}{e_1-e_2}

\title{Skein Theoretic Approach to Yang-Baxter Homology}

\author{Mohamed Elhamdadi} 
\address{Department of Mathematics, 
	University of South Florida, Tampa, FL 33620, U.S.A.} 
\email{emohamed@math.usf.edu} 

\author{Masahico Saito} 
\address{Department of Mathematics, 
	University of South Florida, Tampa, FL 33620, U.S.A.} 
\email{saito@usf.edu} 

\author{Emanuele Zappala} 
\address{Department of Mathematics, 
	University of South Florida, Tampa, FL 33620, U.S.A.} 
\email{zae@mail.usf.edu}

\begin{document}

%
%
%
%
%
%
%

\maketitle

\begin{abstract}


We introduce skein theoretic techniques to compute the Yang-Baxter (YB) homology and cohomology groups of the R-matrix corresponding to the Jones polynomial. More specifically, we 
show that the YB operator $R$ for Jones, normalized for homology, 
admits a skein decomposition $R = I + \beta\alpha$, where $\alpha: V^{\otimes 2} \rightarrow k$ is a ``cup'' pairing map and $\beta: k \rightarrow V^{\otimes 2}$ is a ``cap'' copairing map, 
and  differentials in the chain complex associated to $R$ can be decomposed into 
horizontal tensor concatenations of cups and caps.
We apply our skein theoretic approach to determine 
the second and third YB homology groups, confirming a conjecture of Przytycki and Wang. Further, we 
compute the cohomology groups of $R$, and provide computations in higher dimensions that yield some annihilations of submodules.
\end{abstract}

\section{Introduction}
Yang-Baxter (YB) operators, i.e. solutions of the Yang-Baxter equation (YBE), have been first introduced and studied in Statistical Mechanics \cite{Jim}, due to their connection to scattering and integrable systems. They
have also played a central role in low-dimensional topology, where they are used to construct link and 3-manifold quantum invariants \cite{Tur,Tur2}, via representations of quantum groups and certain kinds of ribbon categories. Also, the study of set-theoretic YB operators has lead to introducing cocycle invariants of links from algebraic structures such as {\it quandles} \cite{CJKLS}.

Subsequently, homology theories for the Yang-Baxter equation have been developed and studied
in relation to deformation theories \cite{Eis},  and with applications to knot invariants
generalizing the notion of quandle cocycle invariants~\cite{CES}.  In particular, in \cite{CES} a (co)homology theory for set-theoretic Yang-Baxter equation was developed. More specifically, a Yang-Baxter set is a pair $(X,R)$, where $X$ is a set and $R:X \times X \rightarrow X \times X$ is an invertible map satisfying the equation
$$(R \times 1)(1 \times R)(R \times 1)=(1 \times R)(R \times 1)(1 \times R),$$ 
with $1:X \rightarrow X$ denoting the identity map. 
 Given a Yang-Baxter set $(X,R)$, in \cite{CES} a (co)chain complex associated to $R$ has been introduced, whose $2$-cocycles were used to produce invariants of classical and virtual knots. 
In \cite{LeVe,Prz}, this homology theory was generalized to the YB operators ($R$-matrix) on tensor products of vector spaces (or modules):
$R: V \otimes V \rightarrow  V \otimes V$ satisfying $$(R \otimes \mathbbm{1})(\mathbbm{1} \otimes R) (R \otimes \mathbbm{1})=(\mathbbm{1} \otimes R) (R \otimes \mathbbm{1})(\mathbbm{1} \otimes R),$$  
where $\mathbbm{1}: V \rightarrow V$ is the identity map.
They also provided an alternative diagrammatic of the chain maps that unifies the set-theoretic and tensor YBEs. 
In \cite{PW} 
it was shown that for set-theoretical case,  the two homology theories are equivalent.  

Also in  \cite{PW}, a family of YB operators corresponding to the Jones and HOMFLYPT polynomials was considered. The original matrices were normalized in order to define chain complexes. Computer based results and a conjecture related to the R-matrix corresponding to Jones polynomial were presented. In \cite{PW2}, the second homology group
for the matrices corresponding to the HOMFLYPT polynomial has been computed. 

The main purpose of this paper is to develop techniques to compute (co)homology groups of the YB operator corresponding to Jones polynomial. We do so by simplifying the differentials $d_n$ defining its YB homology. More specifically (Theorem~\ref{thm:diff}) we decompose $d_n$ in terms of sums of simpler maps $g_k, g'_k, h_k$ and $h'_k$  (see Figure~\ref{leftgens} for a diagrammatic interpretation), by using the skein relation satisfied by the normalized matrix $R$. As an application we explicitly give the corresponding decompositions of the differentials $d_n$ for $n = 2,3,4$ and compute the corresponding matrices and their Smith normal form, using preliminary results on  $g_k, g'_k, h_k$ and $h'_k$ . 

This article is organized as follows. In Section~\ref{sec:pre}, we recall the definitions of Yang-Baxter differentials and related homology, normalized Kauffman bracket R-matrix and a conjecture of Przytycki and Wang.  In Section~\ref{sec:KauffR} we show that $R$ satisfies the skein relation $R = \mathbbm{1} + \alpha\beta$, where $\alpha$ is a pairing  diagrammatically represented by a cup, and $\beta$ is a copairing represented by a cap.
We set up a diagrammatic formalism that will be used in the rest of the paper to simplify proofs and computations.
In particular, we apply it to show that the normalized matrix $R$ satisfies the YBE.
Section~\ref{sec:skeindifferential} is the central part of the article. Here we show that the differentials corresponding to $R$, defining YB homology, can be decomposed as tensor products of certain generating maps, that are represented by horizontal concatenations of corresponding diagrams.
We therefore proceed, in Section~\ref{sec:lowdim}, to apply the skein theoretic decomposition of the differentials to compute the homology of $R$ in low dimensions, confirming the case $n = 3$ of Przytycki-Wang conjecture. In Section~\ref{sec:YBcohomology} we dualize our methodology to compute low dimensional cohomology groups of $R$. Finally, in Section~\ref{sec:higherdim} we study the torsion of the homology groups in higher dimensions.  Precisely, 
for $X=(V, R)$ where $R$ is the normalized matrix in \cite{PW} on a rank 2 module $V$,  
we show that for every odd $n$, there exists a rank 2 submodule of $H_n(X)$ that is annihilated by 
multiplication by $y^4-1$.  For every even $n$, we show that there exists a rank one submodule $K_1$ of $Z_n(X)$ that is in the boundary group $B_n(X)$,
and a rank one submodule $K_2$ that is annihilated by 
multiplication by $y^2-1$.  
Some of the proofs are deferred to the appendices.

\section{Preliminary}\label{sec:pre}

\subsection{Yang-Baxter operators and their normalization}

Let $V$ be a $k$-module over a unital ring $k$. 
We say that $M$ is a right (resp. left)  $V$-module 
if there is a $k$-morphism ({\it action}) $\mu_\ell: M \otimes V\rightarrow M$ (resp. $\mu_r: V \otimes M \rightarrow M$), this unusual choice of conventions comes from the fact that the right action appears in the left differential of YB homology, while the left action appears in the right differential.  
Below when we focus on the right action,  we drop the subscript and use $\mu=\mu_{\ell}$.
In this paper we exclusively consider the ground ring $k=\Q[y, y^{-1}]$ and $M=k$ 
with {\it trivial} actions $\mu_\ell (a \otimes x)=a= \mu_r (x \otimes a)$ for all $a \in M$, $x \in V$. 

An invertible morphism 
$R: V \otimes V \rightarrow  V \otimes V$ is called {\it Yang-Baxter (YB) operator}, or an {\it R-matrix}, if it satisfies $$(R \otimes \mathbbm{1})(\mathbbm{1} \otimes R) (R \otimes \mathbbm{1})=(\mathbbm{1} \otimes R) (R \otimes \mathbbm{1})(\mathbbm{1} \otimes R).$$

An R-matrix $R$ is said to satisfy the {\it left wall condition} if it satisfies
$$ \mu_\ell (\mu_\ell \otimes \mathbbm{1}_V ) (\mathbbm{1}_M \otimes R) = 
 \mu_\ell (\mu_\ell \otimes \mathbbm{1}_V ) ,$$
 and the {\it right wall condition} is defined similarly. An R-matrix satisfies the wall condition if it satisfies both left and right wall conditions. The left wall condition is depicted in Figure~\ref{curtaincondition}. In the figure, the $V$-module $M$ is represented by the shaded vertical line, and thin lines represent $V$. The crossing at the left figure represents  the map $R$, and the map $\lambda_\ell: M \otimes V \rightarrow M$ is represented by merging two (shaded and thin) lines.

\begin{figure}[htb]
\begin{center}
\includegraphics[width=1.2in]{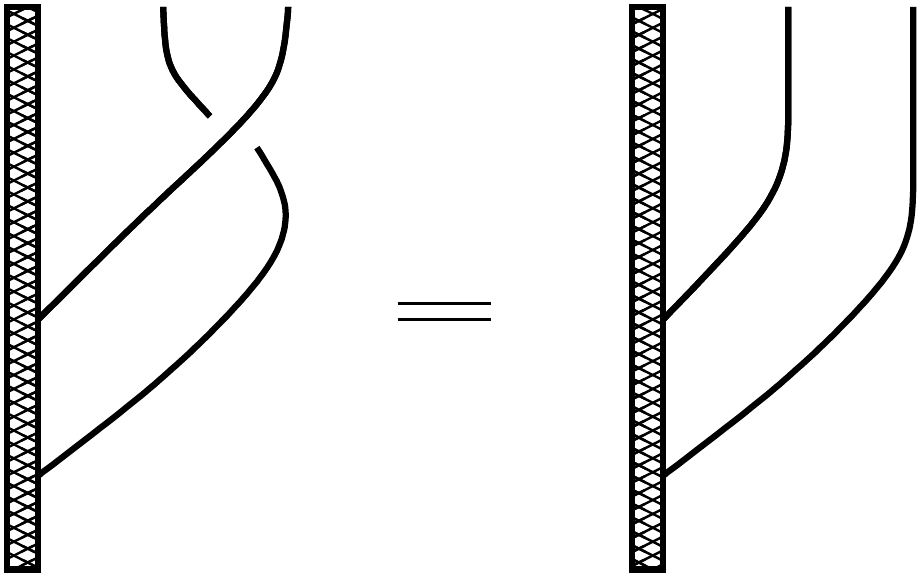}
\end{center}
\caption{The left  wall condition}
\label{curtaincondition}
\end{figure}

Let $R$ be an R-matrix over $V \otimes V$.
It is observed in \cite{PW} that for  the trivial action $\mu_\ell(1 \otimes e) = 1 = \mu_r (e \otimes 1)$ for every basis vector $e$  to  satisfy the wall condition is  that the matrix is {\it column unital}, i.e., the sum of entries of each column is $1$. 
We call the procedure of making a matrix column unital the {\it normalization}.

\subsection{Yang-Baxter  differentials}\label{sec:YBchain}

Let the maps 
$$d^\ell_{i, n}, d^r_{i, n}  \in {\rm Hom}(M \otimes V^{\otimes n } \otimes M, M \otimes V^{\otimes (n-1) }  \otimes M ) $$ be defined by
\begin{eqnarray*}
\lefteqn{d_{i, n}^\ell}\ \ &=&
 (\mu_\ell\otimes \mathbbm{1}^n)\circ (R\otimes \mathbbm{1}^{n-2})\circ \cdots   \\
 && \cdots \circ(\mathbbm{1}^{i-3}\otimes R\otimes\mathbbm{1}^{n-i+1})\circ(\mathbbm{1}^{i-2}\otimes R\otimes\mathbbm{1}^{n-i})\\
\lefteqn{d_{i, n}^r}\ \ &=&
(\mathbbm{1}^{n}\otimes \mu_r)\circ (\mathbbm{1}^{n-2}\otimes R)\circ\cdots\\
&& \cdots\circ (\mathbbm{1}^{n-i+1}\otimes R\otimes \mathbbm{1}^{i-3})\circ (\mathbbm{1}^{n-i}\otimes R\otimes \mathbbm{1}^{i-2}).
\end{eqnarray*}
We also use the notations 
$d^s_n=\sum_i d^s_{i,n}$ for $s=l, r$.
These maps are diagrammatically represented in Figure~\ref{curtain}.

The differentials of the Yang-Baxter homology is defined by
$$d_n= \sum_{i=1}^n (-1)^i [\ d_{i, n}^\ell - d_{i, n}^r \ ] . $$

\begin{figure}[htb]
\begin{center}
\includegraphics[width=3in]{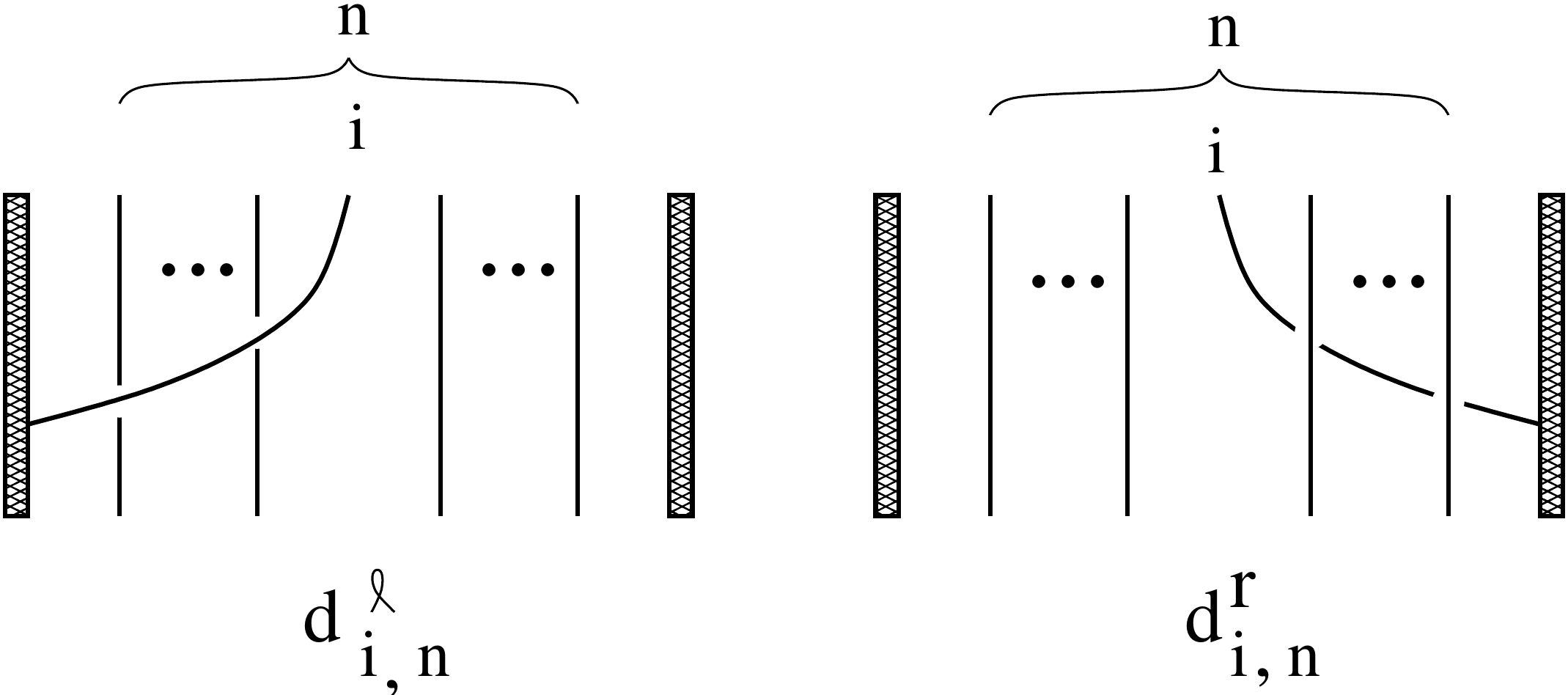}
\end{center}
\caption{Left and right curtain maps}
\label{curtain}
\end{figure}

In \cite{LeVe,PW}, it was proved that 
$d^2=0$,
so that $d$ defines a chain complex called Yang-Baxter (YB) homology.  The proof of this fact can be observed by diagrammatic means, and is illustrated in Figure~\ref{diff}.


\begin{figure}[htb]
\begin{center}
\includegraphics[width=3in]{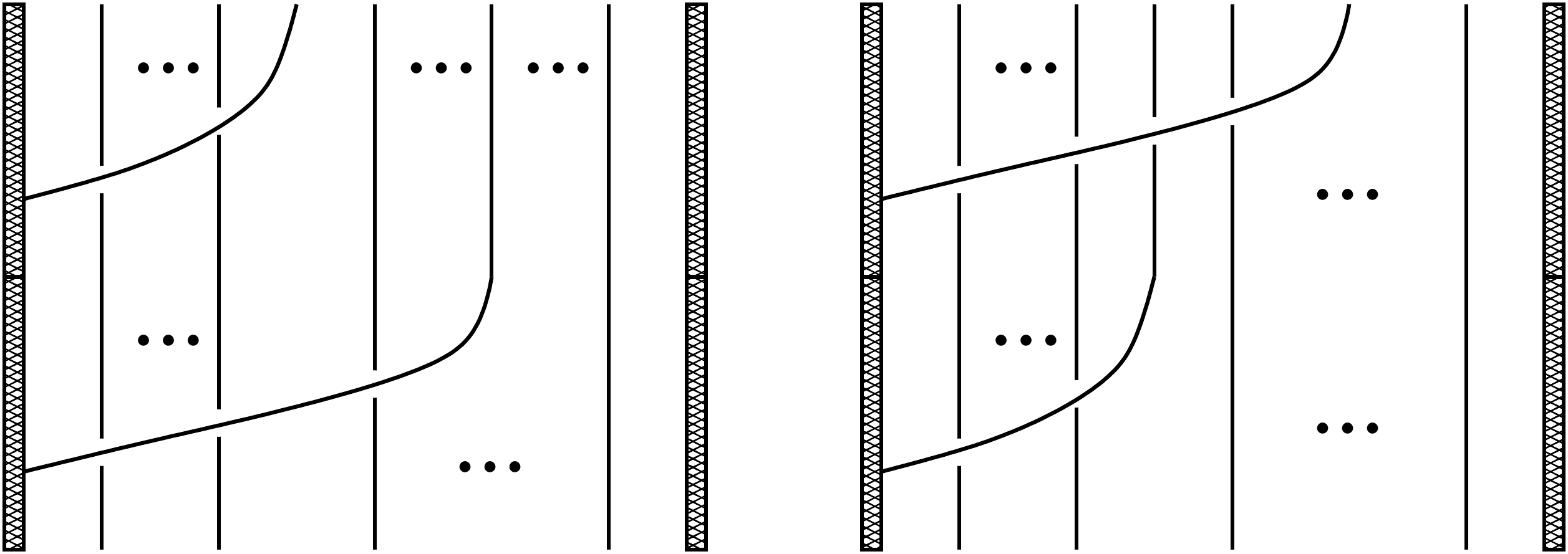}
\end{center}
\caption{$d^2=0$}
\label{diff}
\end{figure}

\subsection{Kauffman bracket (Jones) R-matrix}\label{sec:R}

For the R-matrix $R'$  that produces  the Jones polynomial, the normalization of making $R'$ column unital has been performed in \cite{PW}, where it is shown that  the normalized matrix takes the from

 $$R=\left( 
\begin{array}{cccc}
1 & 0 & 0 & 0 \\
0 & 1-y^2 & 1 & 0 \\
0 & y^2 & 0 & 0 \\
0&0&0&1
\end{array}
\right).$$
For the rest of the paper we focus on this specific R-matrix.

Let $e_1$, $e_2$ be the basis elements of the rank 2 free $k$-module $V$  with respect to  which the map $R$ is the above matrix.
Specifically, the rows and columns of $R$ are for the basis elements 
$e_1 \otimes e_1$, $e_1 \otimes e_2$, $e_2 \otimes e_1$, $e_2 \otimes e_2$ in this order.
For this specific R-matrix , 
a conjecture on YB homology groups is stated in  \cite{PW} as follows, where $X = (V,R)$.

\begin{conjecture}
$H_n(X)=k^2 \oplus ( k/(1-y^2) )^{a_n} \oplus ( k/(1-y^4) )^{s_{n-2}} $ 
where $s_n=\sum_{i=1}^{n+1}f_i$ is the partial sum of Fibonacci sequence with $f_1=f_2=1$ and $a_n$ is given by $a_1=0$ and $2^n=2+a_{n-1} + s_{n-3}+a_n+s_{n-2}$.
\end{conjecture}

We note that $s_0$ is defined to be $1$, though this may not be explicit in \cite{PW}.

\section{Skein for the  normalized Kauffman bracket $R$-matrix}\label{sec:KauffR}

In this section we establish a skein relation for the normalized R-matrix defined above,
and define diagrammatic representations. 
It is not a priori the case that a normalized matrix $R$ of a YB solution $R'$ is a YB solution, but this fact is proved in \cite{PW}. 
We use the skein relation to provide a diagrammatic proof of this fact. We also provide lemmas on maps that appear in the skein that will be used in later sections.

\begin{lemma}
We have  $R=I + J$ where $I$ denotes the identity matrix and 
$J=\beta \alpha$, where 
$\alpha : V \otimes V \rightarrow k$ and 
$\beta : k \rightarrow V \otimes V$ are defined by 
 \begin{eqnarray*}
 \alpha (e_1 \otimes e_1) & = & \alpha(e_2 \otimes e_2)\quad = \quad 0, \\
\alpha(e_1 \otimes e_2) &=& -y , \\
\alpha(e_2 \otimes e_1) &=& y^{-1},  \\
\beta(1) &=&  y (e_1 \otimes e_2 - e_2 \otimes e_1).
\end{eqnarray*}

\end{lemma}

\begin{proof}
Let 
$J=\left( 
\begin{array}{cccc}
0 & 0 & 0 & 0 \\
0 & -y^2 & 1 & 0 \\
0 & y^2 & -1  & 0 \\
0&0&0&0
\end{array}
\right)$.
Then we have $R=I + J$. 
Furthermore $J$ is written as 
$J=(0, y, -y, 0 )^T \cdot (0, -y, y^{-1}, 0 )$ where ${}^T$ denotes the transpose.
This means that $J$ is the composition of a pairing $\alpha: V \otimes V \rightarrow k$ represented by $(0, -y, y^{-1}, 0 )$ and the copairing $\beta: k \rightarrow V \otimes V$ 
represented by $(0, y, -y, 0 )^T $, and the result follows.
\end{proof}

Let $\xi, \zeta: V \rightarrow V$ be defined by 
$\xi(e_1)=y^2 e_1$, $\xi(e_2)=e_2$, $\zeta(e_1)=e_1$, $\zeta(e_2)=y^2 e_2$.
Then we have the following remark and lemma by straightforward calculations.

\begin{remark}\label{rem:maps}
	{\rm 
The maps $\alpha$,  $\xi$ and $\zeta$ can be written as follows:
\begin{eqnarray*}
\alpha (e_i\otimes e_j) &=& (-1)^i(1-\delta_{ij}) y^{j-i},\\
\zeta(e_i) &=& y^{2i-2} e_i,\\
\xi(e_i) &=& y^{4-2i} e_i,
\end{eqnarray*}	
where $i=1,2$ and $\delta_{ij}$ denotes the Kronecker's delta.
}
\end{remark}
\begin{lemma}
We have 
\begin{eqnarray*}
( \mathbbm 1  \otimes \beta)( \alpha \otimes \mathbbm 1) &=& \xi ,  \\ 
(\beta  \otimes  \mathbbm 1) (\mathbbm 1 \otimes \alpha ) &=& \zeta  .
\end{eqnarray*} 
\end{lemma}

\begin{figure}[htb]
\begin{center}
\includegraphics[width=2in]{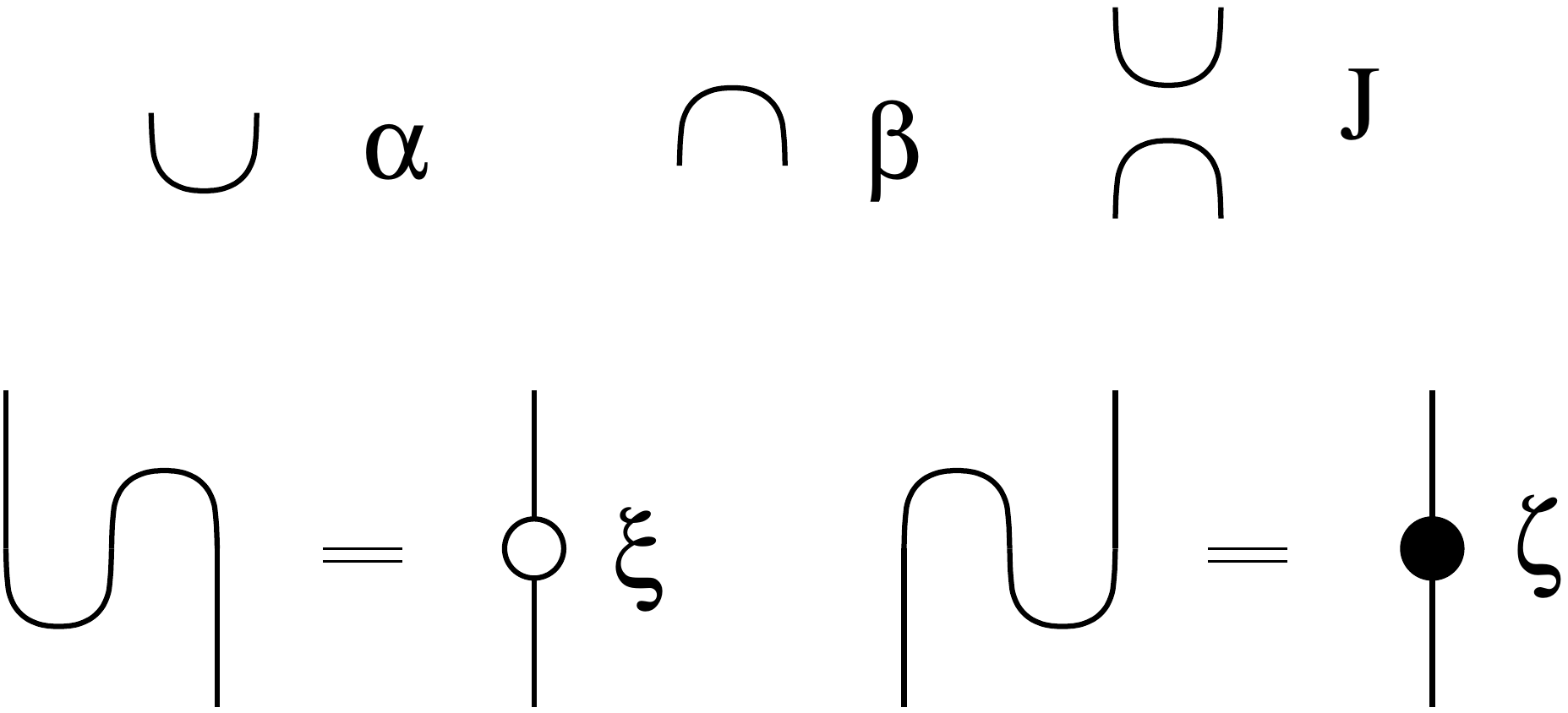}
\end{center}
\caption{Cup, cap and zig-zag maps}
\label{cupcap}
\end{figure}

Straightforward computations also show the following lemma, with diagrams found in Figure~\ref{lemmas}.

\begin{lemma}\label{lem:loop}
We have the following:
\begin{enumerate}
\item $\alpha \beta=-(y^2+1)$,
\item $\xi \zeta = \zeta \xi = y^2 \mathbbm{1}$,
\item $ \alpha (\lambda_\ell \otimes \mathbbm 1 ) = \mu \zeta $.
\end{enumerate}
\end{lemma}

\begin{figure}[htb]
\begin{center}
\includegraphics[width=5in]{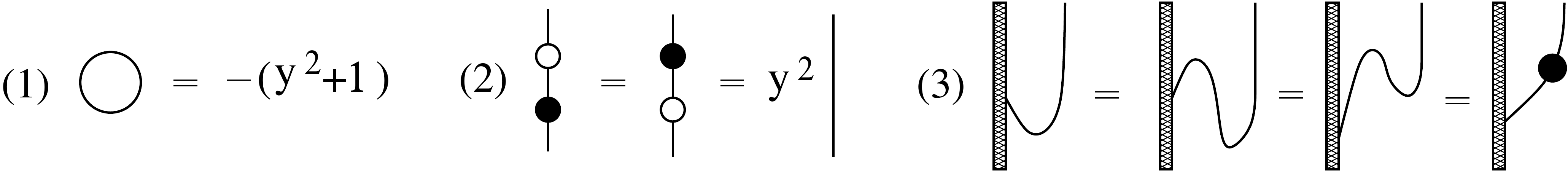}
\end{center}
\caption{A few identities}
\label{lemmas}
\end{figure}

%

It is not clear, a priori, whether the (column) normalized R-matrix satisfies the YBE.
We show below that our diagrammatic approach provides a proof for this fact.

\begin{figure}[htb]
\begin{center}
\includegraphics[width=3.5in]{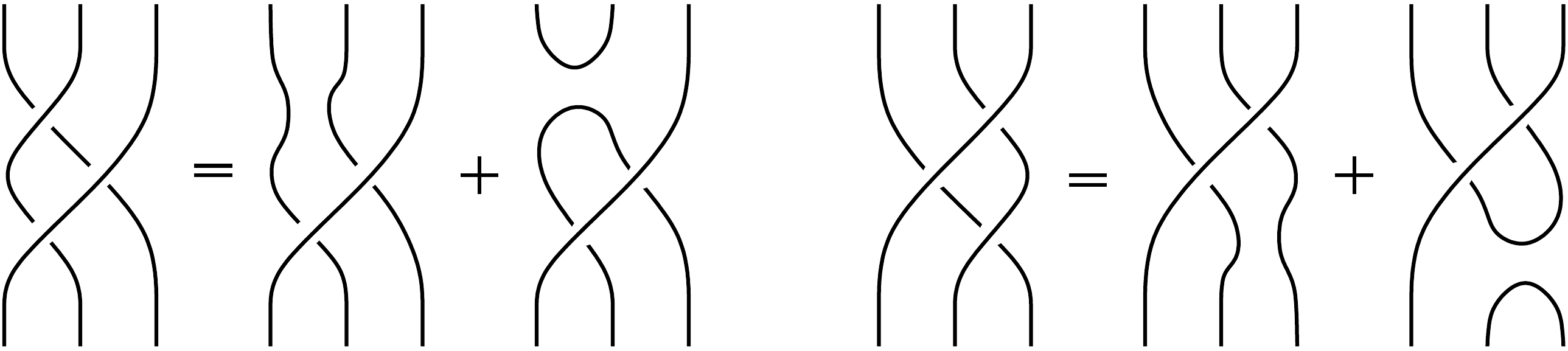}
\end{center}
\caption{The normalized R-matrix satisfies the YBE}
\label{YBE}
\end{figure}

\begin{figure}[htb]
\begin{center}
\includegraphics[width=3.5in]{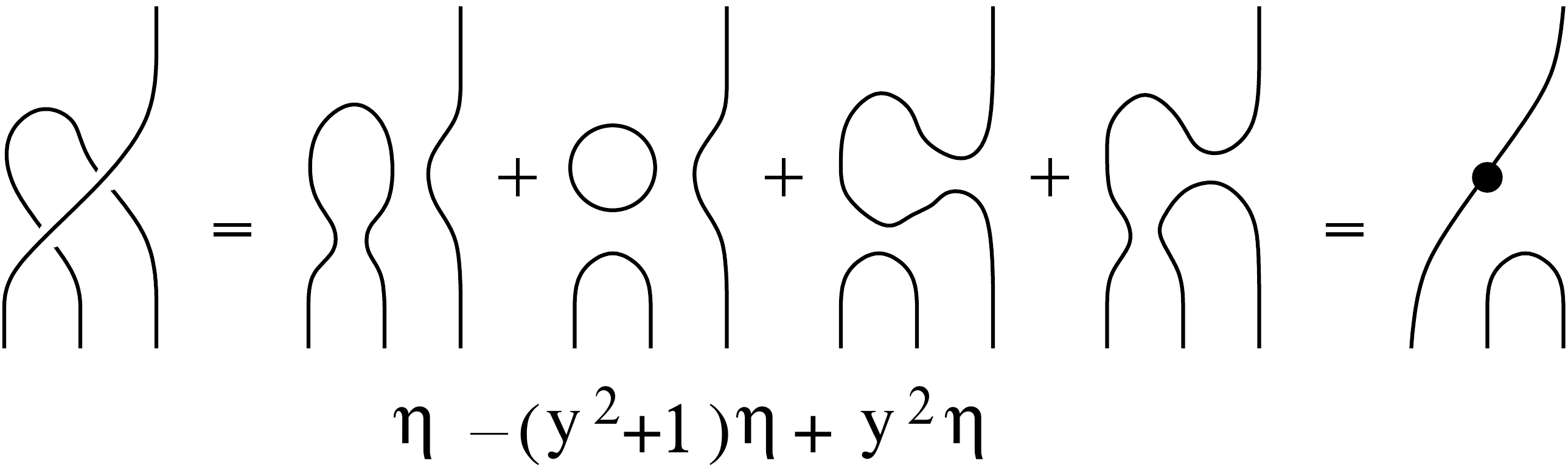}
\end{center}
\caption{The normalized R-matrix satisfies the YBE (continued, $\eta=\beta \otimes \mathbbm{1}$) }
\label{YBE2}
\end{figure}

\begin{theorem}\label{thm:R}
The normalized R-matrix  $R$ satisfies the YBE.
\end{theorem}

\begin{proof}
	Applying the skein relation to the Reidemeister move III as in Figure~\ref{YBE} shows that it is enough to prove that the second summand on the RHS of the first equality is the same as the second summand on the RHS of the second equality. 
	Let us denote these maps $V^{\otimes 3} \longrightarrow V^{\otimes 3}$ by $\Theta_1$ and $\Theta_2$ respectively. We can further simplify $\Theta_i$ as the products
	\begin{eqnarray*}
		\Theta_1 &=& \Phi_1 ( \alpha \otimes  \mathbbm{1}  ) ,\\
		\Theta_2 &=& ( \mathbbm{1} \otimes \beta)  \Phi_2, 
		\end{eqnarray*}  
where $\Phi_1: V \longrightarrow V^{\otimes 3}$ is depicted in the left side of Figure~\ref{YBE2}, and a similar definition is given for $\Phi_2 : V^{\otimes 3} \longrightarrow V$. 
Using again the skein relation to the diagrammatic definitions of $\Phi_i$, $i= 1,2$ and the definitions of cup/cap and zig-zag maps as in Figure~\ref{YBE2} (where Lemma~\ref{lem:loop} (2) was used) we see that
	\begin{eqnarray*}
		\Phi_1 &=& \zeta \otimes \beta,\\
		\Phi_2 &=&  \alpha \otimes \zeta,
		\end{eqnarray*}
and we obtain 
$$ 	\Theta_1 = \Phi_1 ( \alpha \otimes  \mathbbm{1}  ) = ( \zeta \otimes \beta) ( \alpha \otimes  \mathbbm{1}  ) =  ( \mathbbm{1} \otimes \beta) (  \alpha \otimes \zeta)=  ( \mathbbm{1} \otimes \beta)  \Phi_2 =	\Theta_2 $$
	which concludes the proof.
\end{proof}

\begin{remark}
{\rm
The modified bracket skein relation $R=I + J$ and Theorem~\ref{thm:R} implies that this $R$-matrix defines a braid group representation that factors through a {\it skew} Temperley-Lieb algebra ${\rm STL}_n$ defined as follows.

For a positive integer $n$,  ${\rm STL}_n$ is a $k$-algebra ($k=\Q[y, y^{-1}]$) generated by $h_i$, $i=1, \ldots, n-1$ and relations 
\begin{eqnarray*}
& & h_i h_i = - (y^2+1) h_i \quad i=1, \ldots, n-1, \\
& & h_i h_{i+ 1} h_i = y^2 h_{i+1} \quad  i=1, \ldots, n-2, \\
& & h_i h_{i- 1} h_i = y^2 h_{i-1} \quad  i=2, \ldots, n-1, \\
& & h_i h_j = h_j h_i \quad |i - j | >1 , \ i, j +1, \ldots n-1 .
\end{eqnarray*}

Diagrammatic representation of $h_i$ is to place a pair of cup and cap at the $i^{\rm th}$ and $(i+1)^{\rm st}$ positions as for the Temperley-Lieb algebra (Figure~\ref{STL} (A)). 
The first relation follows from Lemma~\ref{lem:loop} (1). 
The second  relation is depicted in Figure~\ref{STL} (B) which follows from Lemma~\ref{lem:loop} (2). 
These diagrammatic correspondence and computations imply that the assignment
$h_i \mapsto \mathbbm{1}^{i-1} \otimes (\beta \alpha) \otimes  \mathbbm{1}^{n-i-1}$ induces a homomorphism 
${\rm STL}_n \rightarrow {\rm Aut} (V^{\otimes n} ) $ where $V$ is a rank 2 $k$-module.
The skein $R=I+J$ where $J=\beta\alpha$ implies that the braid group representation induced from $R$, 
$B_n \rightarrow  {\rm Aut} (V^{\otimes n} ) $ defined by $\sigma_i \mapsto  \mathbbm{1}^{i-1}\otimes  R\otimes  \mathbbm{1}^{n-i-1} $.

\begin{figure}[htb]
\begin{center}
\includegraphics[width=2in]{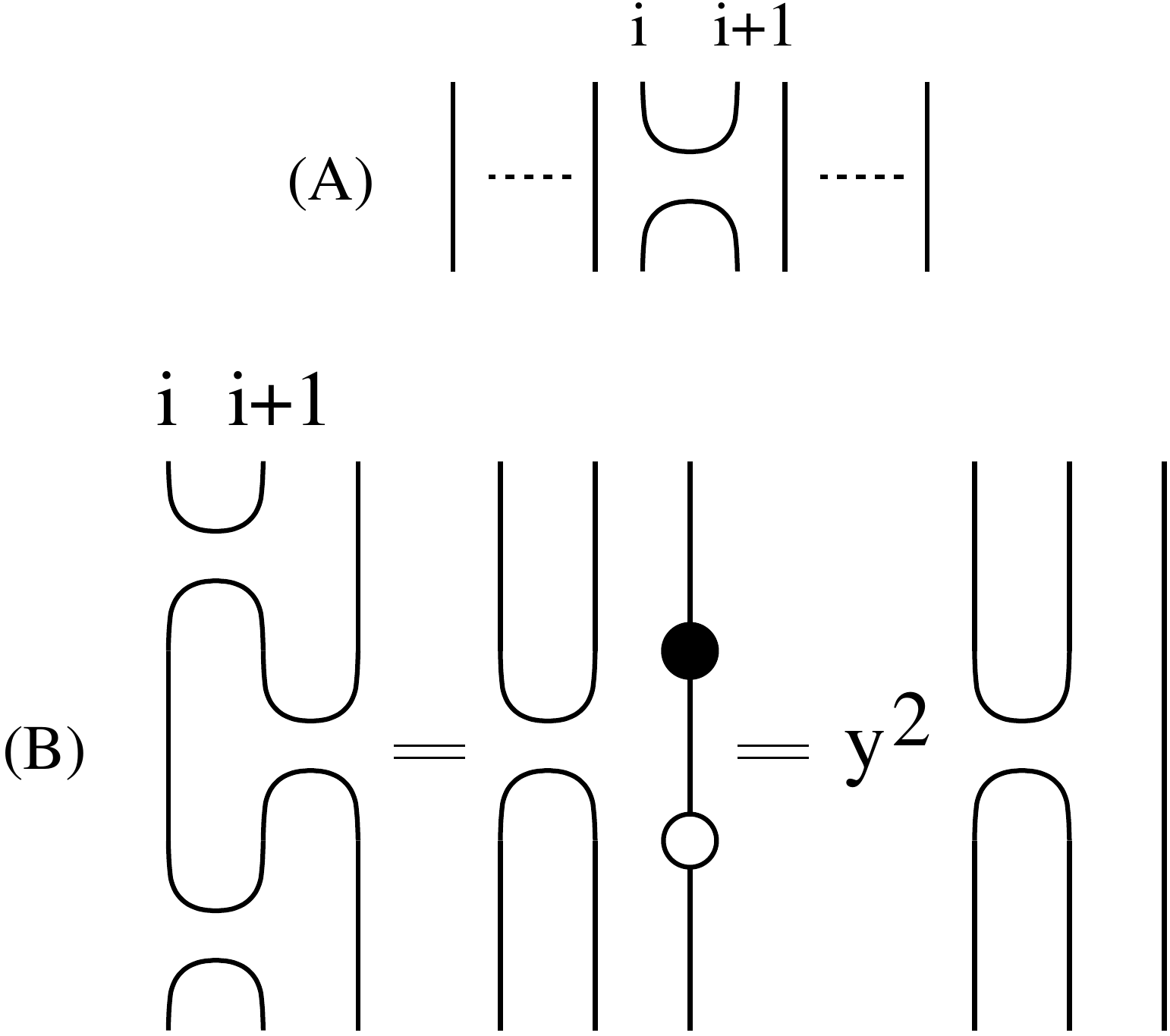}
\end{center}
\caption{Skew Temperley-Lieb algebra generator and relation }
\label{STL}
\end{figure}

}
\end{remark}

\begin{figure}[htb]
\begin{center}
\includegraphics[width=2.5in]{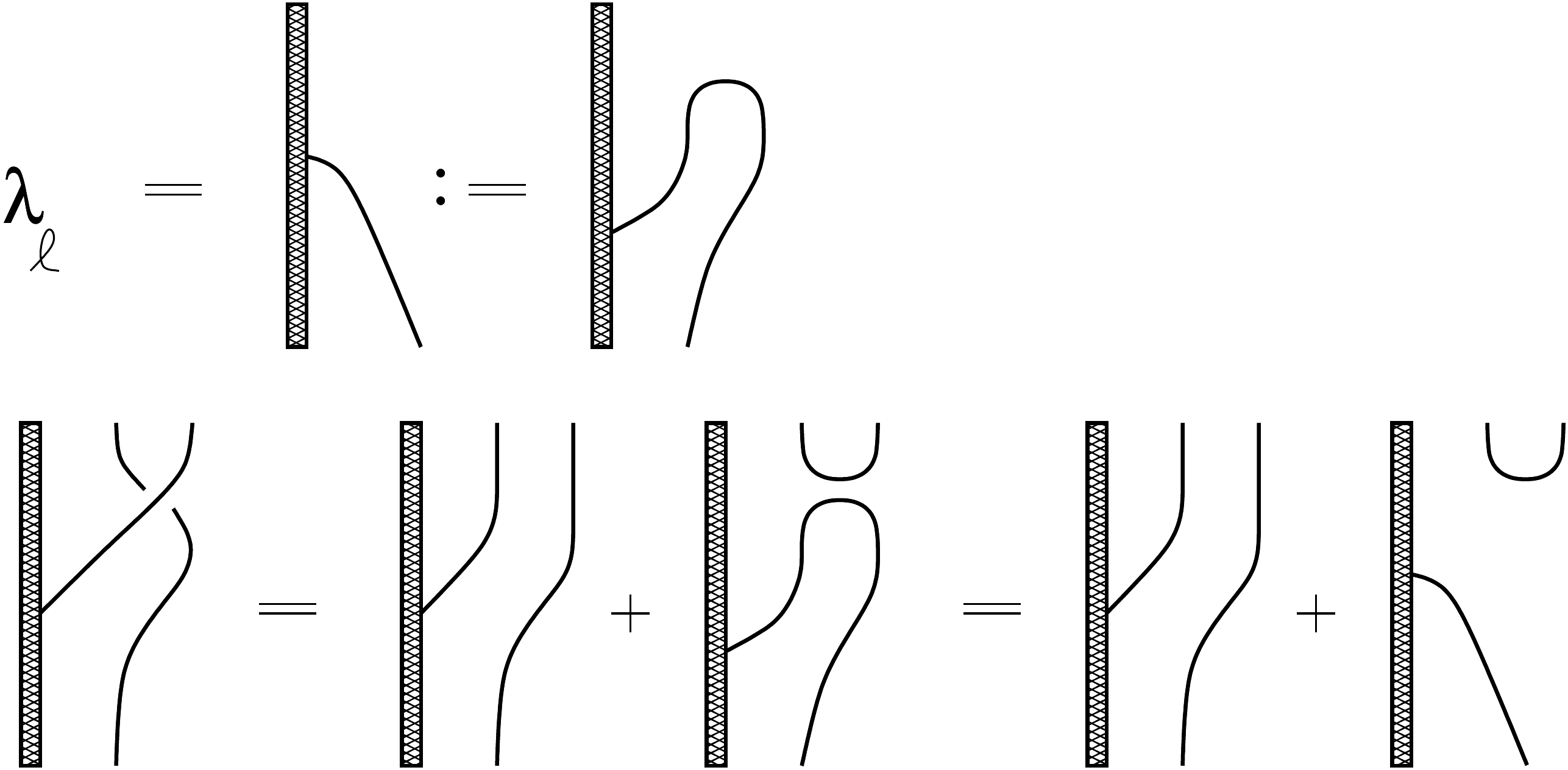}
\end{center}
\caption{Definition of $\lambda_\ell$ and how it appears in skein}
\label{lambda}
\end{figure}

\begin{figure}[htb]
\begin{center}
\includegraphics[width=2in]{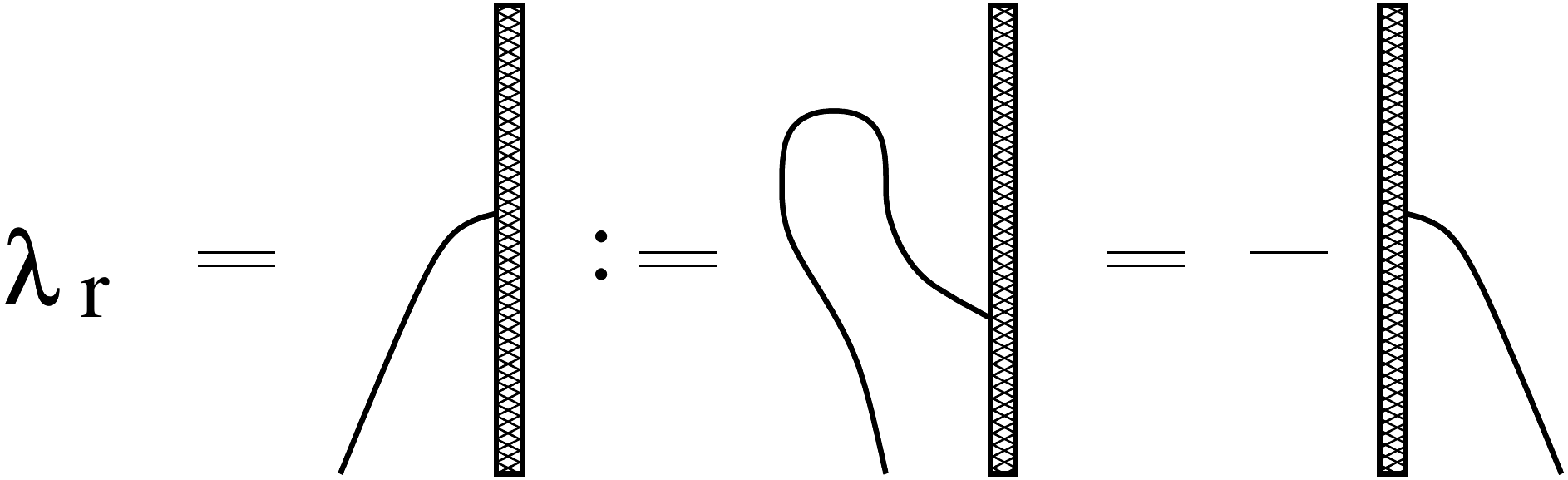}
\end{center}
\caption{Definition of $\lambda_r$ which is the negative of $\lambda_\ell$}
\label{lambdar}
\end{figure}

The left and right curtain maps 
$\lambda_\ell , \lambda_r : k \rightarrow V$ are defined 
as in Figures~\ref{lambda} 
and \ref{lambdar}, respectively, and computed as 
$$\lambda_\ell(1)=y (e_2 - e_1) \quad {\rm resp.} \quad
\lambda_r (1) = y (e_1 - e_2) .$$
In particular, as depicted  in   Figure~\ref{lambdar}, we obtain that
 $\lambda_r = - \lambda_\ell $.

\begin{figure}[htb]
\begin{center}
\includegraphics[width=2in]{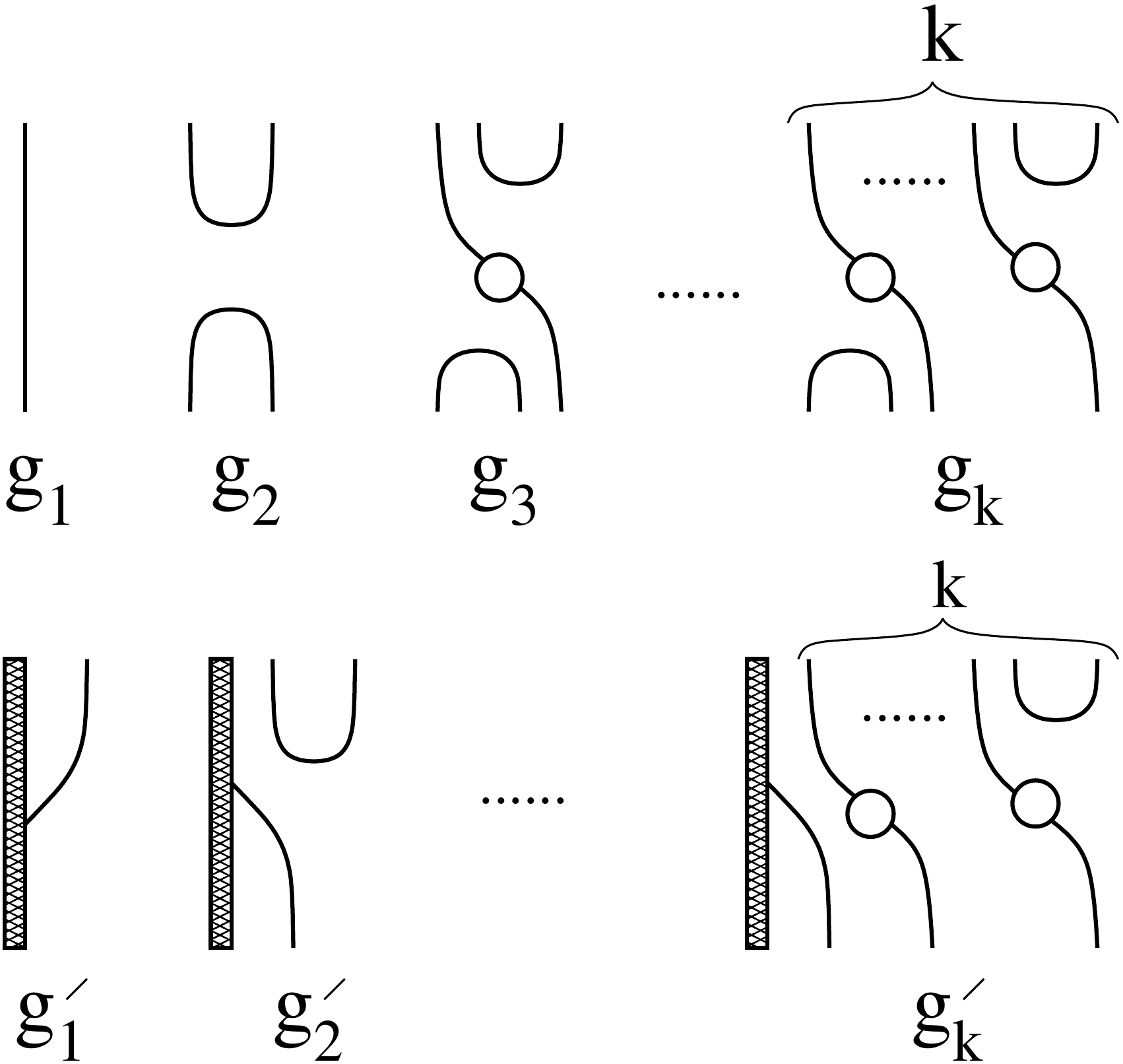}
\end{center}
\caption{Left generators}
\label{leftgens}
\end{figure}

\section{Skein theoretic decomposition of Yang-Baxter differentials}\label{sec:skeindifferential}

In this section we compute a general decomposition of the differentials $d_n$, when the Yang-Baxter operator $R$ satisfies the skein relation $R=I + J$.

We introduce the maps $\{g_i\}$, $\{g'_j\}$ corresponding to the diagrams depicted in Figure~\ref{leftgens}. 
Similar definitions and considerations hold for $\{h_k\}$ and $\{h'_t\}$, right differential generators. 
Our main objective is to show that the Yang-Baxter differentials corresponding to $R=I + J$ can be written 
as described in Theorem~\ref{thm:diff}.




\begin{theorem} \label{thm:diff}
The left differential is written as 
$$d^\ell_n = \sum_{S(n)} \  [ \   g_{i_0} ' g_{i_1}^{k(1)}  \cdots  g_{i_h}^{k(h)} g_{1}^{2k}   \  ]  
$$
where
 $$S(n)=\{ (i_0 , \ldots,  i_h ; k(1) , \ldots , k(h) )\  |\  i_h \neq 1,\ 
i_0 + i_1^{k(1)} + \cdots + i_h^{k(h)} + 2k =n  \}  .$$ 

%
The $n^{\rm th}$ Yang-Baxter differential decomposes as
 $$d_n=-  \sum [ \    g_{i_0} ' g_{i_1}^{k(1)}  \cdots  g_{i_h}^{k(h)} g_{1}^{2k}   \ ]  
 +  \sum [ \   h_1^{2k} h_{i_k}^{k(h)} \cdots g_{i_1}^{k(1)}  h_{i_0} '    \ ] $$ if  $n$ is odd, and 
 $$d_n=  \sum [ \    g_{i_0} ' g_{i_1}^{k(1)}  \cdots  g_{i_h}^{k(h)} g_{1}^{2k}   \ ]  
 +  \sum [ \   h_1^{2k} h_{i_k}^{k(h)} \cdots g_{i_1}^{k(1)}  h_{i_0} '    \ ] $$  if  $n$ is even, where all sums are over $S(n)$. 
\end{theorem}

The proof is given in Appendix~\ref{apdx:diff}. We note that the exponents of the identity maps $g_1(=h_1$) are even. Therefore it is not the case that all possible horizontal concatenations of generating maps appear in the decompositions.

Specializing to the case in which $R$ is the normalized R-matrix corresponding to Jones polynomial, as in Section~\ref{sec:KauffR}, we compute the generators $g_i$ and $g'_i$. 

\begin{lemma}\label{lem:generators}
On basis vectors, the left generators $g_k$ and $g'_k$, with $k\geq 2 $, satisfy
\begin{eqnarray*}
g_k(e_{i_1}\otimes \cdots \otimes e_{i_k}) &=& \theta(i_1,\ldots, i_k) (e_1\otimes e_2-e_2\otimes e_1)\bigotimes_{t=1}^{k-2} e_{i_t},\\ 
g_k' (e_{i_1}\otimes \cdots \otimes e_{i_k}) &=& \theta(i_1,\ldots, i_k) (e_2- e_1)\bigotimes_{t=1}^{k-2} e_{i_t},
\end{eqnarray*}
where $\theta(i_1,\ldots, i_k) := (-1)^{i_{k-1}}(1-\delta_{i_{k-1}i_k})y^{i_k-i_{k-1}}y^{4(k-2)-2\sum_{j=1}^{k-2}i_j}y$. 
\end{lemma}

\begin{proof}
The lemma is a direct computation using the definition of $g_k$ and $g'_k$, along with Remark~\ref{rem:maps}. We leave this verification to the reader. 
\end{proof}

Similarly, we compute the right generators $h_l$ and $h'_l$ with coefficients $\tau(i_1,\ldots, i_k)$.

\begin{lemma}\label{lem:rightgenerators}
	On basis vectors, the right generators $h_l$ and $h'_l$, with $l \geq 2$, satisfy
	\begin{eqnarray*}
	h_l(e_{i_1}\otimes \cdots \otimes e_{i_l}) &=& \tau(i_1,  \ldots, i_l) \bigotimes_{j=3}^{l} e_{i_j}\otimes (e_1\otimes e_2-e_2\otimes e_1),\\
		h'_l(e_{i_1}\otimes \cdots \otimes e_{i_l}) &=& \tau(i_1,  \ldots, i_l) \bigotimes_{j=3}^{l} e_{i_j}\otimes (e_1-e_2),
	\end{eqnarray*}
	where $\tau(i_1, \ldots, i_l) := (-1)^{i_1}(1-\delta_{i_1,i_2})y^{i_2-i_1}y^{2\sum_{t=3}^li_t}y^{2(l-2)}y$.
\end{lemma}

\section{Low-dimensional differentials and homology groups}\label{sec:lowdim}

 In this section we utilize the skein theoretic procedure described in Section~\ref{sec:KauffR},
 that is similar to the Kauffman bracket,  to simplify the differentials and compute Yang-Baxter homology groups in low dimensions, for the normalized  Yang-Baxter matrix $R$. We apply diagrammatic arguments in addition to appealing directly to Theorem~\ref{thm:diff} in order to better illustrate the procedure. 
 Let $X=(V, R)$ be as in Subsection~\ref{sec:R}.

\subsection{The first differential}

By definition the first differential is $d_1: V(=k \otimes V=V \otimes k) \rightarrow k $,
$d_1=- (\mu_\ell - \mu_r)=0$. 
Hence $H_0(X)=0$ and $Z_1(X)=V$.

\subsection{The second differential}

Since left and right coactions have opposite signs, it follows that the second differential is identically null, as the following lemma shows.

\begin{lemma}\label{lem:triviald2}
We have $d_2=0$.
\end{lemma}

\begin{proof}
Diagrammatic computations are depicted in Figure~\ref{diff2} where
the relation $\lambda_r = - \lambda_\ell $ depicted in Figure~\ref{lambdar} is used at the last step. Alternatively, using Theorem~\ref{thm:diff} and Lemma~\ref{lem:generators} we have that $d_2 = g'_2 + h'_2 = 0$.
\end{proof}

It follows that $H_1(X)=V$ from $Z_1(X)=V$ as noted in the preceding subsection.

\begin{figure}[htb]
\begin{center}
\includegraphics[width=2.5in]{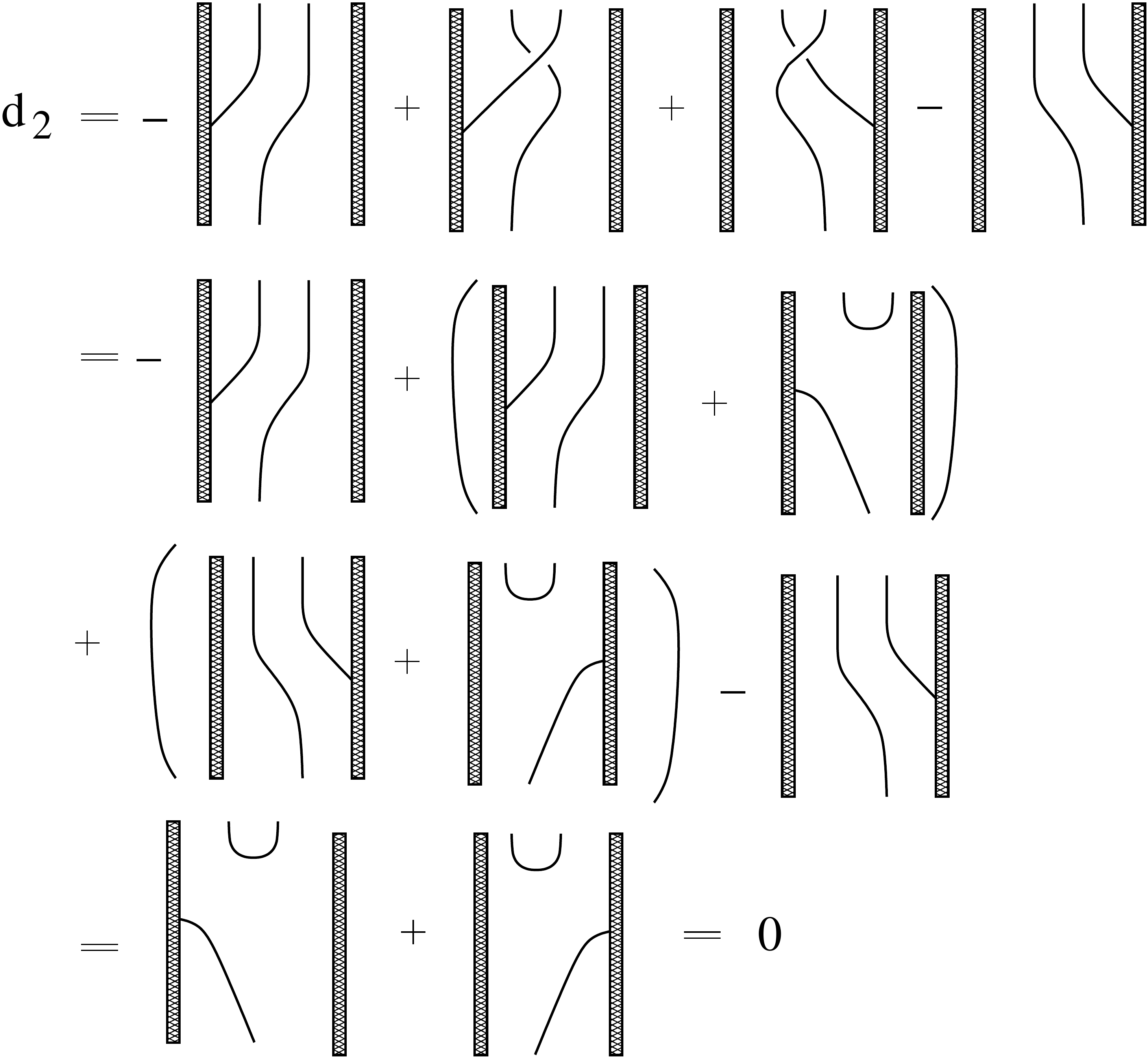}
\end{center}
\caption{Second differential}
\label{diff2}
\end{figure}

\subsection{The third differential}

We now proceed to computing the third differential. Again, we provide a direct diagrammatic interpretation although the next lemma easily follows from Theorem~\ref{thm:diff}. 

\begin{lemma} \label{lem:d31}
The left and right third differentials are given in terms of generators by 
\begin{eqnarray*}
d^\ell_{3} & = &  - g_1' g_1^2 - g_1' g_2 - g_3' , \\
d^r_{3} & = &    h_1^2h_1' +  h_2h_1'  + h_3' .
\end{eqnarray*}
\end{lemma}

\begin{proof}
Diagrammatic computations in 
Figure~\ref{d3left} give the left differential, where the left walls are abbreviated.  The right differential is similar.
Together we obtain $d_3$ as depicted in Figure~\ref{d3}. 
\end{proof}

\begin{figure}[htb]\label{fig:thirddiff}
\begin{center}
\includegraphics[width=2.5in]{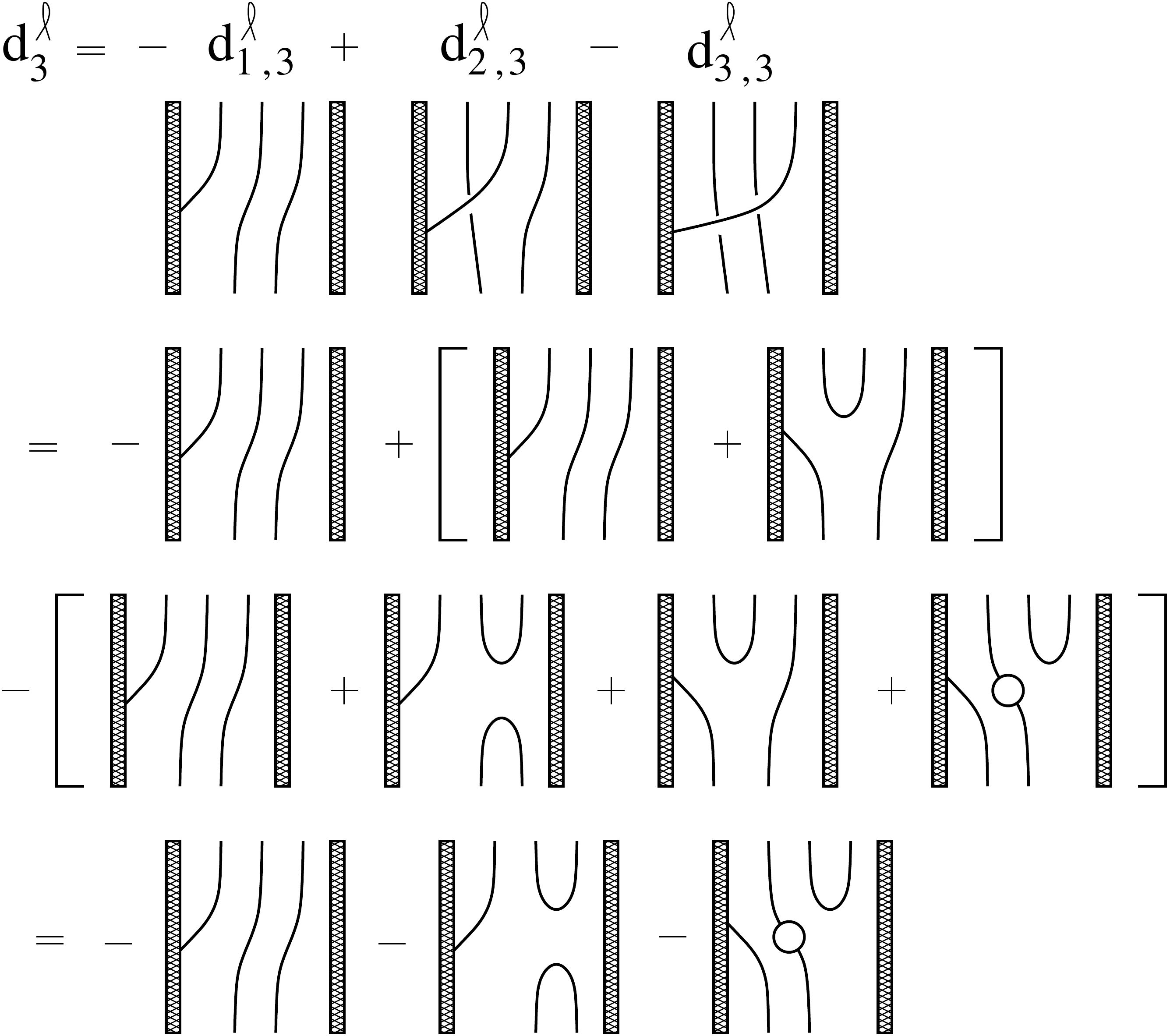}
\end{center}
\caption{Third left differential}
\label{d3left}
\end{figure}

\begin{figure}[htb]
\begin{center}
\includegraphics[width=4in]{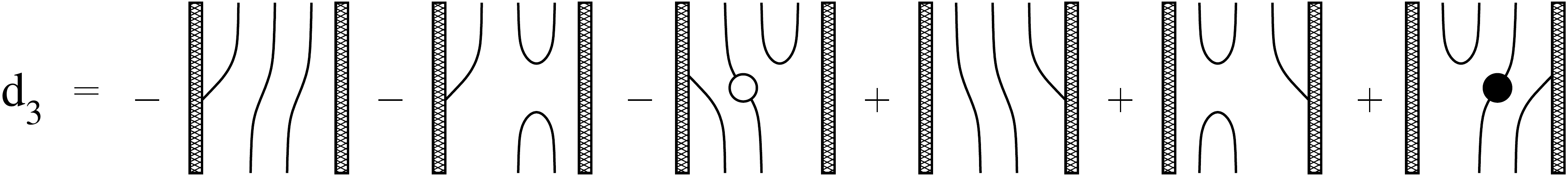}
\end{center}
\caption{Third  differential}
\label{d3}
\end{figure}

\begin{lemma}  \label{lem:d32}
The third differential is given on basis vectors by 
\begin{eqnarray*}
d_3( e_1 \otimes e_1 \otimes e_2 ) &=& (1-y^4) e_1 \otimes e_1  + (y^2 -1) e_1 \otimes e_2
+ y^2( y^2 - 1 ) e_2 \otimes e_1  , \\
d_3( e_1 \otimes e_2 \otimes e_2 ) &=& (1-y^2) e_1 \otimes e_2  + y^2(1-y^2 ) e_2 \otimes e_1
+ ( y^4 - 1 ) e_2 \otimes e_2  ,
\end{eqnarray*}
and $d_3( e_i \otimes e_j \otimes e_k )=0$ otherwise.
\end{lemma}

\begin{proof}
On basis vectors $e_i\otimes e_j\otimes e_k$, with $i,j,k = 1,2$, from Lemma~\ref{lem:generators} and Lemma~\ref{lem:d31} we have
\begin{eqnarray*}
	d_3 (e_i\otimes e_j\otimes e_k) &=& -e_j\otimes e_k + (-1)^{j+1} (1-\delta_{jk})y^{k-j+1}(e_1\otimes e_2 - e_2\otimes e_1)\\
	&& + (-1)^{j+1}(1-\delta_{jk})y^{k-j-2i+5}(e_2-e_1)\otimes e_i + e_i\otimes e_j\\
	&& + (-1)^i(1-\delta_{ij})y^{j-i+1}(e_1\otimes e_2-e_2\otimes e_1)\\
	&& + (-1)^i(1-\delta_{ij})y^{j-i+2k-1} e_k\otimes (e_1 - e_2) . 
 \end{eqnarray*}
When $i=j=k$, since $1-\delta_{ij} = 1-\delta_{jk} = 0$, we have 
$$
d_3 (e_i\otimes e_j\otimes e_k) = -e_i\otimes e_i + e_i\otimes e_i = 0. 
$$
Let us consider the case $i = k \neq j$. Since 
$i+j=3$, 
we have $y^{i-j+1} = y^{2i-2}$, $y^{j-i} = y^{3-2i}$ and $(-1)^i = (-1)^{j+1}$ 
so that 
\begin{eqnarray*}
	d_3 (e_i\otimes e_j\otimes e_i) &=& -e_j\otimes e_i + (-1)^iy^{2i-2}(e_1\otimes e_2-e_2\otimes e_1)\\
	&& + (-1)^iy^2(e_2\otimes e_i - e_1\otimes e_i) + e_i\otimes e_j \\
	&& + (-1)^iy^{4-2i}(e_1\otimes e_2 - e_2\otimes e_1) \\
	&& + (-1)^i y^2(e_i\otimes e_1 - e_i\otimes e_2).
	\end{eqnarray*}
 Distinguishing the two cases $i=1$ and $i=2$ we easily see that either way $d_3(e_i\otimes e_j\otimes e_i) = 0$. Finally, we consider $ i = j \neq k$ and $i\neq j = k$. 
 In the first case, since $y^{5-3i + k} = y^{8-4i} $ 
 and $y^{k-i+1} = y^{4-2i}$ we have
 \begin{eqnarray*}
 	d_3(e_i\otimes e_i\otimes e_k) &=& -e_i\otimes e_k + (-1)^{i+1} y^{4-2i}(e_1\otimes e_2 - e_2\otimes e_1) \\
 	&& + (-1)^{i+1}y^{8-4i}(e_2\otimes e_i - e_1\otimes e_i) + e_i\otimes e_i,
 	\end{eqnarray*}
 which is readily seen to be zero when $i = 2$ and equal to $(1-y^4) e_1 \otimes e_1  + (y^2 -1) e_1 \otimes e_2
 + y^2( y^2 - 1 ) e_2 \otimes e_1$ when $i = 1$. Similarly, $i\neq j = k$ gives zero when $k = 1$ and $(1-y^2) e_1 \otimes e_2  + y^2(1-y^2 ) e_2 \otimes e_1
 + ( y^4 - 1 ) e_2 \otimes e_2$ when $k = 2$. This concludes the proof of the lemma. 
\end{proof}

We now compute the second homology group. 
\begin{theorem}\label{thm:secondhomology}
We have
$H_2(X)= k^2 \oplus k / (y^2 -1 ) \oplus k / (y^4 -1) $.
\end{theorem}

\begin{proof}
From Lemma~\ref{lem:d32},
in matrix form with columns for $e_1 \otimes e_1 \otimes e_2$ and $e_1 \otimes e_2 \otimes e_2$,
and rows for $e_1 \otimes e_1 $, $e_1 \otimes e_2 $, $e_2 \otimes e_1 $, $e_2 \otimes e_2 $
in these orders, the  matrix below.
$$
\left[ \begin{array}{rr} 1 - y^4 & 0 \\
y^2 -1 & 1 - y^2 \\
y^4 - y^2 & y^2 - y^4 \\
0 & y^4 -1 
\end{array}
\right]
$$
Then 
changes of bases are performed as follows.
$$
\left[ \begin{array}{rr} 1 - y^4 & 1 - y^4 \\
y^2 -1 & 0 \\
y^4 - y^2 & 0 \\
0 & y^4 -1 
\end{array}
\right]
\left[ \begin{array}{rr} 1 - y^4 & 0 \\
y^2 -1 & 0 \\
0 & 0 \\
0 & y^4 -1 
\end{array}
\right]
\left[ \begin{array}{rr} 0 & 0 \\
y^2 -1 & 0 \\
0 & 0 \\
0 & y^4 -1 
\end{array}
\right]
$$
The right-most matrix represents the group as stated.
\end{proof}

\subsection{The fourth differential}
Next we compute the fourth differential. As before we have the following.

\begin{lemma}
The left and right fourth differentials are given in terms of generators by 
\begin{eqnarray*}
d^\ell_{4} & = &   g_1' g_1 g_2 + g_1' g_3 + g_2' g_1^2 + g_2' g_2 + g_3' , \\
d^r_{4} & = &   h_2 h_1 h_1' + h_3 h_1' +  h_1^2h_2' +  h_2h_2'  + h_3' .
\end{eqnarray*}
\end{lemma}

The diagrammatic representation of $d_4$ is found in Figure~\ref{d4}.

\begin{figure}[htb]
	\begin{center}
		\includegraphics[width=4in]{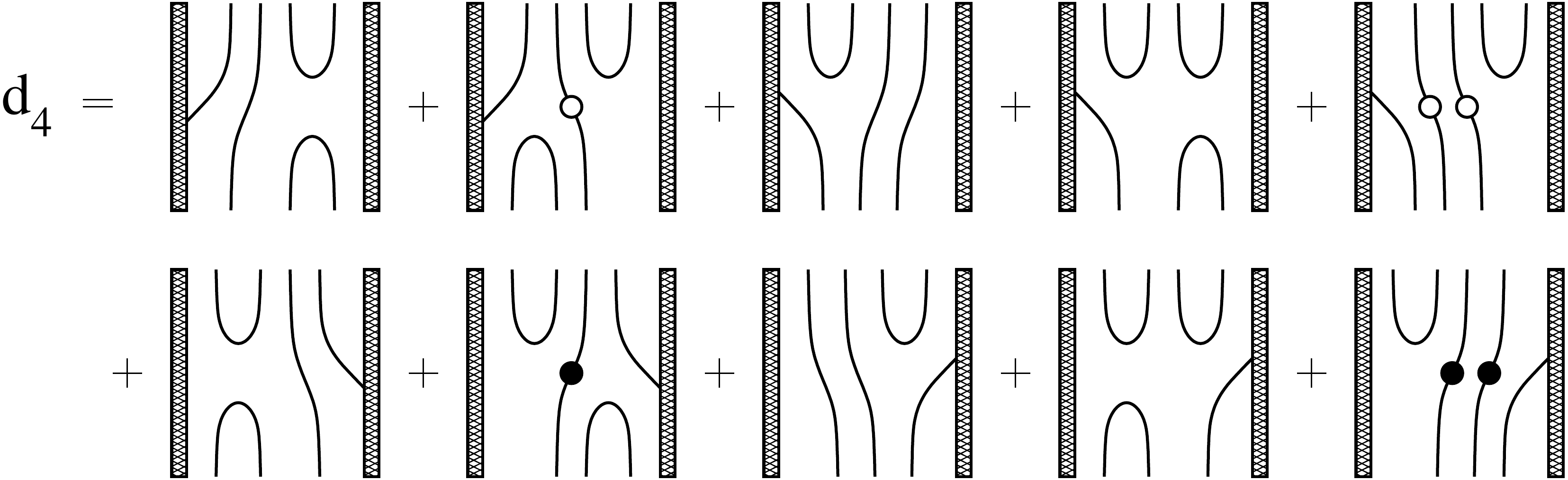}
	\end{center}
	\caption{The fourth differential}
	\label{d4}
\end{figure}

\begin{lemma}\label{lem:d4}
	The matrix form of $d_4$, with respect to the bases of $V^{\otimes 3}$ and $V^{\otimes 4}$ in lexicographic order with respect to the indices, is given by 
	$$
	\left[ \begin{array}{cccccccc} 
	0 & 0 & 0 & 0   & 0 &&&\\
	y^6-y^2&&y^2-y^4&&y^4-y^6&&&\\
	1-y^4&&y^2 -1&& y^4 -y^2&&&\\
	&& 0 && 0 &&&\\
	&& 0 && 0 &&&\\
	&&y^4 -y^2&&y^2-y^4&y^4-y^2&y^2-y^4&\\
	&&1-y^2&&y^2-y^4& 0 &y^4 -1&\\
	&&&&&y^2-y^4&y^4-y^6&y^6-y^2\\
	&&&&& 0 & 0 &\\
	&&&&y^4 -1&1-y^2&y^2-y^4&\\
	&&&&& 0 & 0 &\\
	&&&&&y^2-1&y^4-y^2&1-y^4\\
	&&&&& 0 &&\\
	&&&&& 0 &&\\
	&&&&& 0 &&\\
	&&&&& 0 &&
	\end{array}
	\right] 
	$$
	\end{lemma}
	For exposition the matrix is transposed, so that the eight columns correspond to 
	$e_i \otimes e_j \otimes e_k $ and  the rows correspond to 
	$e_i \otimes e_j \otimes e_k \otimes e_\ell $.
	For example, the second row represents that 
	\begin{eqnarray*}
\lefteqn{	d_4( e_1 \otimes e_1 \otimes e_1 \otimes e_2)}\\
&=&
(y^6-y^2) e_1 \otimes e_1 \otimes e_1 
	+ ( y^2-y^4) e_1 \otimes e_2 \otimes e_1 
	+ (y^4-y^6)e_2 \otimes e_1 \otimes e_1 . 
	\end{eqnarray*}
	Blank entries represent zeros, though some zeros are given to clarify the positions of entries.
\begin{proof}
	Either by direct diagrammatic manipulation using the skein relation, or using Theorem~\ref{thm:diff}, it follows that $d_4$ is represented diagrammatically as in Figure~\ref{d4}. Using Lemma~\ref{lem:generators} we compute $d_4$ on basis vectors $e_i\otimes e_j\otimes e_k\otimes e_l$ with $i,j,k,l = 1,2$. We have
	\begin{eqnarray*}
	\lefteqn{d_4(e_i\otimes e_j\otimes e_k\otimes e_l)}\\
	&=& \theta(k,l)e_j\otimes (\sym) + \theta(j,k,l) (\sym)\otimes e_j \\
	&& 
	+ \theta(i,j)(\lwall)\otimes e_k\otimes e_l \\
	&&
	 + \theta(i,j)\theta(k,l) (\lwall)\otimes (\sym)\\
	&& + \theta(i,j,k,l)(\lwall)\otimes e_i\otimes e_j \\
	& & + \tau(i,j)(\sym)\otimes e_k\\
	&& + \tau(i,j,k)e_k\otimes (\sym) +  \tau(k,l)e_i\otimes e_j\otimes (\rwall) \\
	&& + \tau(i,j)\tau(k,l) (\sym)\otimes (\rwall) \\
	&& + \tau(i,j,k,l) e_k\otimes e_l\otimes (\rwall).  
	\end{eqnarray*}
The differential can be computed directly from this formula.

We illustrate alternative computations aided by diagrams.
For computing the coefficient of $e_1 \otimes e_1 \otimes e_1$, we observe that the only contributions are given by four maps whose diagrams are depicted in Figure~\ref{d4example}.
For example, the term (2) in Figure~\ref{d4example} represents the map $\lambda_\ell \cdot \xi \cdot \xi \cdot \alpha$, and it is seen from the diagram that the only terms that give non-zero coefficients for 
$e_1 \otimes e_1 \otimes e_1$ are $e_1 \otimes e_1 \otimes e_1\otimes e_2$ and 
$e_1 \otimes e_1 \otimes e_2\otimes e_1$. 
The value for the former is computed as 
$(\lambda_\ell \cdot \xi \cdot \xi \cdot \alpha ) ( e_1 \otimes e_1 \otimes e_1\otimes e_2)=
(-y) \cdot y^2 \cdot y^2 \cdot (-y)$. All the other terms are computed similarly using diagrams.
\end{proof} 

\begin{figure}[htb]
	\begin{center}
		\includegraphics[width=3.5in]{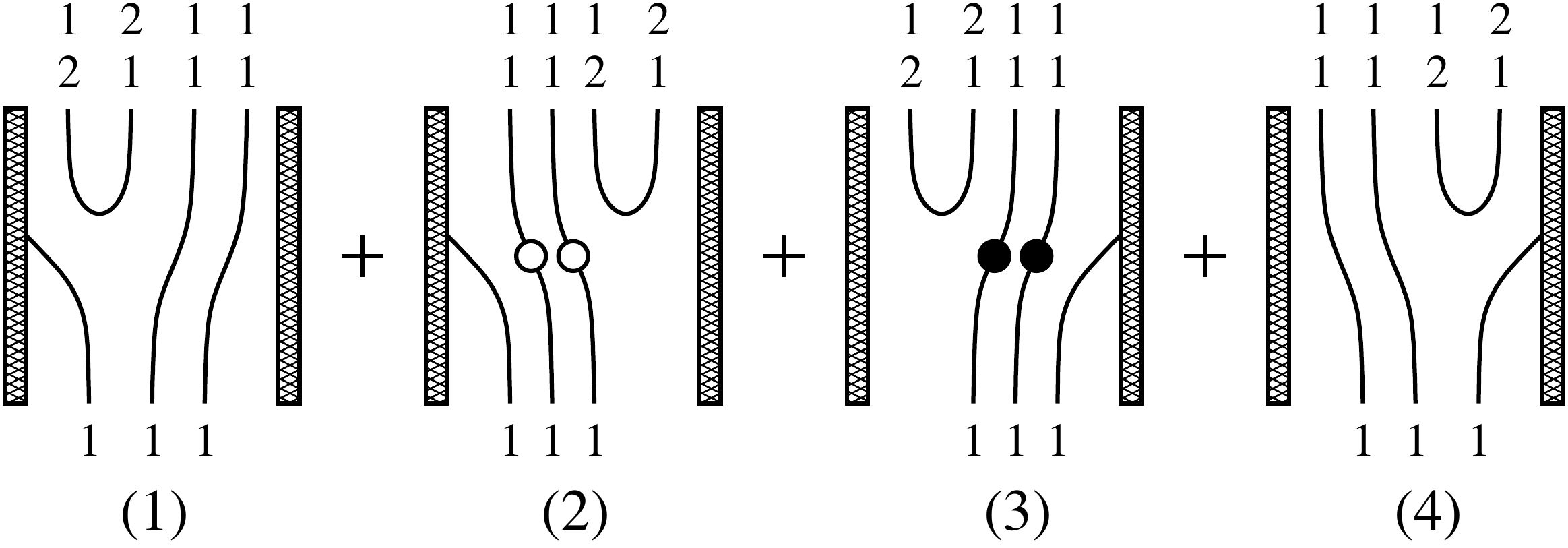}
	\end{center}
	\caption{Sample computations of the fourth differential}
	\label{d4example}
\end{figure}

We can now compute the third homology group of the Yang-Baxter operator $R$.
\begin{theorem}\label{thm:thirdhomology}
	Setting $k := \mathbb{Q}[y,y^{-1}]$ we have $$H_3(X)= k^{\oplus 2}\oplus k/(1-y^2)^{\oplus 2}\oplus k/(1-y^4)^{\oplus 2}.$$
\end{theorem}
\begin{proof}
Applying a sequence of elementary column and row operations to the matrix $d_4$ given in Lemma~\ref{lem:d4}, we obtain the Smith normal form 

$$
\left[ \begin{array}{rrrr} 
	1-y^2&&&\\
	&1-y^2&&\\
	&&1-y^4&\\
	&&&1-y^4
	\end{array}
	\right]
$$
where, for simplicity, we omit the zero columns on the right of the nontrivial diagonal. Since by Lemma~\ref{lem:d32} the kernel of $d_3$ is six-dimensional, the assertion follows.
\end{proof}

\section{Yang-Baxter cohomology}\label{sec:YBcohomology}

		Let $A$ be an abelian group. Then by dualizing the chain complex in Subsection~\ref{sec:YBchain}, we obtain a cohomology theory, called Yang-Baxter cohomoogy, with coefficients in $A$, and differentials written as $\delta^{n+1}: C^n(X; A)  \rightarrow C^{n+1}(X;A)$. We denote  the cohomology groups by 
		$H^n(X;A)$.
We observe that the universal coefficient theorem determines cohomology groups as follows.

\begin{proposition}
Let $k=\Q [y, y^{-1}]$. Then we have 
\begin{eqnarray*}
H^2(X; k) &=&  k^{\oplus 2}, \\
H^3(X; k) &=&  k^{\oplus 2}\oplus k/(1-y^2)\oplus k/(1-y^4). 
\end{eqnarray*}
\end{proposition}

\begin{proof}
The universal coefficient theorem reads 
$$0 \longrightarrow {\rm Ext^1(H_n(X;A), B) \longrightarrow H^n(X; B) 
 \longrightarrow
 {\rm Hom}(H_{n-1}(X,A), B} \longrightarrow 0 . $$
 We take $A=B=k$. Since $H_1(X;k)=k$, and by Theorem~\ref{thm:secondhomology},
 we obtain $H^2_R(X;\mathbbm k)$ as stated.
We have ${\rm Ext}^1(k/ f k , k) \cong k/ f k $ for a Laurent polynomial $f(y)$ in $k$, hence Theorem~\ref{thm:secondhomology} and Theorem~\ref{thm:thirdhomology} 
determine $H^3_R(X;\mathbbm k)$
as stated. 
\end{proof}

\begin{remark}
{\rm
A common argument to show that a $n$-dimensional cohomology group is nontrivial
is to exhibit a non-trivial $n$-cocycle $\theta$ that evaluates an $n$-cycle $x$ non-trivially,
$\theta(x) \neq 0$. We present a diagrammatic method to do this for $H^2$, even though it is already proved, with a hope that a similar technique might prove productive in higher dimensions.

Specifically, we show that $\alpha$ is a non-trivial $2$-cocycle. Figure~\ref{alphacocy} shows that the left differential applied to $\alpha$ gives zero. A similar procedure is used for the right differential. 
By Theorem~\ref{thm:secondhomology} the class represented by $e_1\otimes e_2$ is non-trivial, 
and we observe that $\alpha (e_1\otimes e_2)\neq 0$. 
Hence  $\alpha$ is nontrivial. 

Indeed to show that $\alpha $ is a $2$-cocycle, one could also explicitly compute
$$
\begin{cases}
	(1-y^4) \alpha (e_1\otimes e_1) + (y^2-1) \alpha (e_1\otimes e_2) + y^2(y^2-1) \alpha(e_2\otimes e_1) \\ \quad = 0+(y^2-1)(-y) +  y^2(y^2-1) (y^{-1}=  0, \\
	(1-y^2)  \alpha (e_1\otimes e_2) + y^2(1-y^2)  \alpha(e_2\otimes e_1) + (y^4-1)  \alpha(e_2\otimes e_2) \\ \quad = (1-y^2)  (-y)+  y^2(1-y^2) y^{-1} +0=0
\end{cases}
$$
as desired.
}
\end{remark}

\begin{figure}[htb]
	\begin{center}
		\includegraphics[width=2.5in]{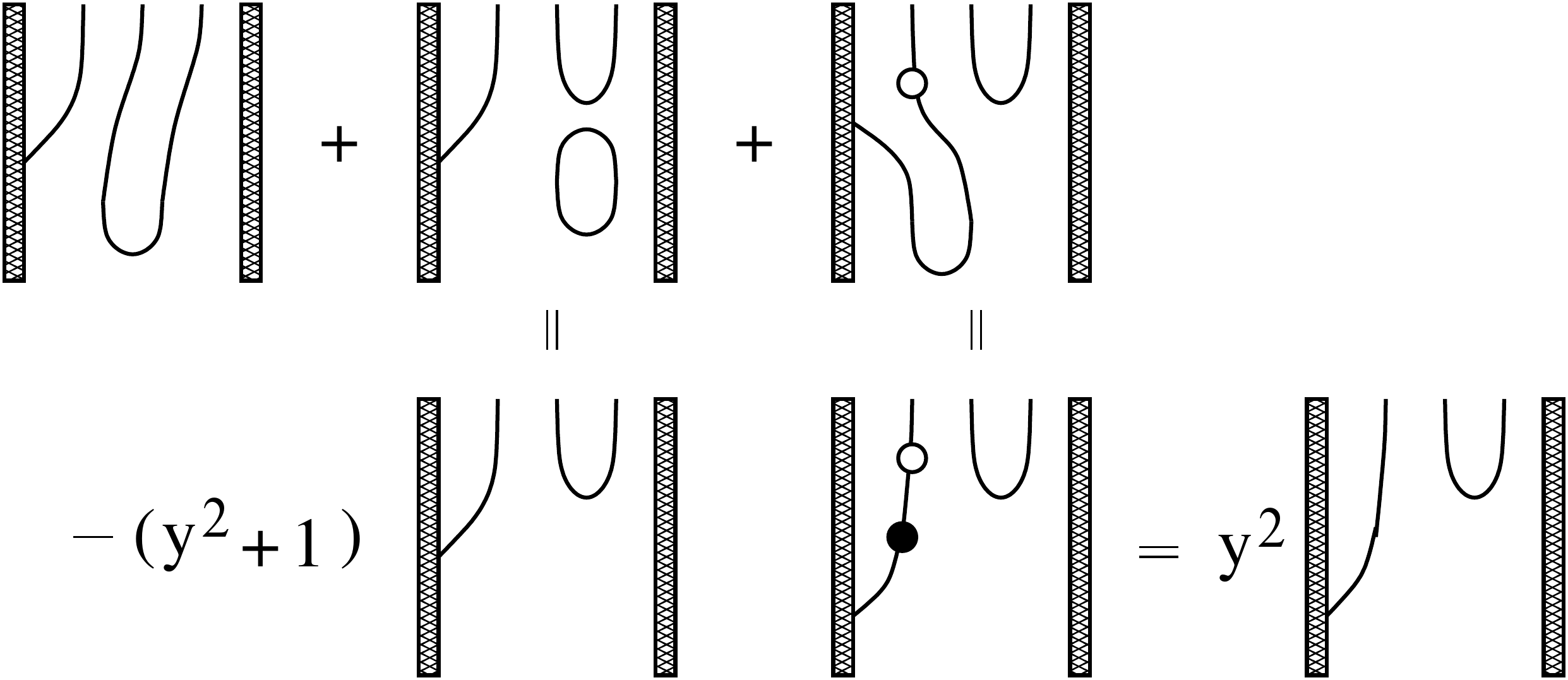}
	\end{center}
	\caption{The pairing $\alpha$ is a 2-cocycle}
	\label{alphacocy}
\end{figure}

\section{Further computations in higher dimensions}\label{sec:higherdim}

In this section we exhibit some diagrammatic computations in higher dimensions and observe annihilations by specific polynomials.

For $i=1, \ldots, n$, 
let $e_{1, i}^n = e_1 \otimes \cdots \otimes e_1 \otimes e_2  \otimes e_1 \otimes \cdots \otimes e_1$ where there are $n$ factors and $e_2$ is at the $i^{\rm th}$ position, and similarly $e_{2, i}^n= e_2 \otimes \cdots \otimes e_2 \otimes e_1 \otimes e_2 \otimes \cdots \otimes e_2$. By convention define
$e_{j, 0}^n =e_j \otimes \cdots \otimes e_j$ for $j=1,2$. 
When understood we surpress the superscript $n$.

\begin{lemma}\label{lem:aj}
For all positive integer $n$, we have $d_n( e_{j, 0}^n)=0$ for $j=1,2$.
\end{lemma}

\begin{proof}
If any of the terms  $g_k, g_k', h_k, h_k'$, say $g_k$ for some $k$,  that appear in Theorem~\ref{thm:diff} contains $\alpha$, then $g_k(e_{j,0})=0$, since $\alpha(e_j \otimes e_j)=0$ for $j=1,2$. 
Hence if $n$ is odd, 
the only non-zero terms in $d_n$  when evaluated by  $e_{j, 0} $ are 
$$(g_1' g_1^{n-1} ) (e_{j, 0}) = e_{j ,0} =
(h_1^{n-1} h_1' ) (e_{j,0} ) $$
and they cancel with opposite signs in $d_n$. 
If $n$ is even, then there is no term without  $\alpha$, hence the image vanishes.
\end{proof}


%
%

\begin{proposition}\label{lem:bj}
For all positive integer $n>3$  and 
$i=1, \ldots, n$,
 we have the following.
 If $n$ is odd, then the coefficient of $e_{1,0}^{n-1}$ and $e_{2,0}^{n-1}$, respectively,  is non-zero for the following:
  $$
\begin{array}{ll}
 d_n(e_{1,1}) = ( y^2) e_{1,0}, &
d_n(e_{1,2i-1}) =  ( 1 - y^{4i-4} )  e_{1, 0} , \\
d_n(e_{1,2i}) =  (  y^{4i -2 } - y^2 ) e_{1, 0} , & 
d_n(e_{1,n}) = ( 1 - y^{2(n-1)} ) e_{1, 0} ,  \\
d_n(e_{2,1})=  (y^{2n-2}-1 )e_{2,0} ,  & 
d_n(e_{2,2i+1}) =(  y^2   - y^{2n-4i-2}  ) e_{2,0} \\
d_n ( e_{2,2i} ) = (y^{2n-4i} -1) e_{2,0} . 
\end{array}
$$
and zero otherwise. If $n$ is even, then the following terms have non-zero coefficients for  $e_{1,0}^{n-1}$ and $e_{2,0}^{n-1}$, and zero otherwise:
$$
\begin{array}{ll}
d_n(e_{1,2i-1}) =(y^{4i-2}  - y^2 ) e_{1,0} , & 
d_n(e_{1,2i})= (1- y^{4(i-1)} ) e_{1,0}  , \\
d_n(e_{2,2i-1}) =(y^2 -  y^{2(n-2i+3)}) e_{2,0} , &
d_n(e_{2,2i})= ( y^{2(n-2i+2)} -1 ) e_{2,0}.
\end{array}
$$
\end{proposition}

\begin{proof}
Since the image of $\beta$ has zero coefficients for $e_{j,0}$ 
for $j=1,2$, 
the only non-zero terms with  $e_{j,0}$  in the image  are 
 the maps described below.
For odd $n$, the maps are 
$g_1' g_1^{n-1} $, $g_{2i+1}'  g_1^{n-2i - 1 }$ for $i=1, \ldots,  (n-1)/2$  ($g_{n}'$ is the case $i=(n-1)/2$),
  $h_1^{2i} h_{n-2i}' $ ($h_{n}'$ is the case $i =0$),  $h_1^{n-1} h_1' $.
 For even $n$,  the maps are
 $g_{2i}'  g_1^{n-2i}$ for $i=1, \ldots,  n/2$ ($g_{n}'$ is the case $i=n/2$),
  $h_1^{2i } h_{n-2i }' $   ($h_{n}'$ is the case $i =0$).
  These maps are represented by diagrams in Figure~\ref{all1} (1)--(4) in this order.
 Note that the requirement in Theorem~\ref{thm:diff} that the exponent of $g_1$ be even leads to the conditions on the parity.


The actual values can also be computed from diagrams, counting  the contributions of $\xi$ and $\zeta$ to the powers of $y$. 

For even $n$ and for $e_{1,0}$, and for $i=1, \ldots, n/2$, we have
\begin{eqnarray*}
(g_{2i}'  g_1^{n-2i} ) ( e_{1, 2i-1} ) &=& (-y) ( y^{-1}) y^{2(2i-2)} e_{1,0}
  \quad = \quad  - y^{4i-4}   e_{1,0} \\
(g_{2i}'  g_1^{n-2i} ) ( e_{1,2i} ) 
&=&   (-y) ( -y  ) y^{2(2i-2)} e_{1,0}
  \quad = \quad     y^{4i-2}  e_{1,0} \\
( h_1^{2i} h_{n-2i}' ) ( e_{1, 2i+1}) 
&=& y (y^{-1}) e_{1,0}  \quad = \quad   1 e_{1,0}  \\ 
( h_1^{2i} h_{n-2i}' ) (  e_{1, 2i+2}) 
&=&y  (-y )  e_{1,0}  \quad = \quad   - y^2  e_{1,0} . 
\end{eqnarray*}
For example, the first tensor factor of  $(g_{2i}'  g_1^{n-2i} ) ( e_{1,2i} ) $  comes from $\lambda_\ell (1)=y(e_2 - e_1)$, and we look at the coefficient of $e_1$, so that this map contributes $(-y)$.
Then $g_{2i}' (e^{2i}_{1,2i}$ contributes $\alpha(e_1 \otimes e_2)=y^{-1}$
 and $\xi(e_1)^{2i-2}=y^{2(2i-2)}$. Other terms are computed similarly with the aid of Figure~\ref{all1}. 
From these we compute
\begin{eqnarray*}
d_n(e_{1,2i-1}) &=& ( -  g_{2i}' g_1^{n-2i}+ h_1^{2i-2} h_{n-2i+2}'  ) ( e_{1, 2i-1} ) \quad =\quad (-y^{4i-4}  +1) e_{1,0}. \\
d_n(e_{1,2i}) &=& ( - g_{2i}' g_1^{n-2i} - h_{2i-2}'  h_{n-2i+2}' ) ( e_{1, 2i} ) \quad =\quad ( y^{4i - 2 } - y^2 ) e_{1,0} . 
\end{eqnarray*} 
Other cases are found in Appendix~\ref{apdx:higher}.
\end{proof}

\begin{figure}[htb]
\begin{center}
\includegraphics[width=4.5in]{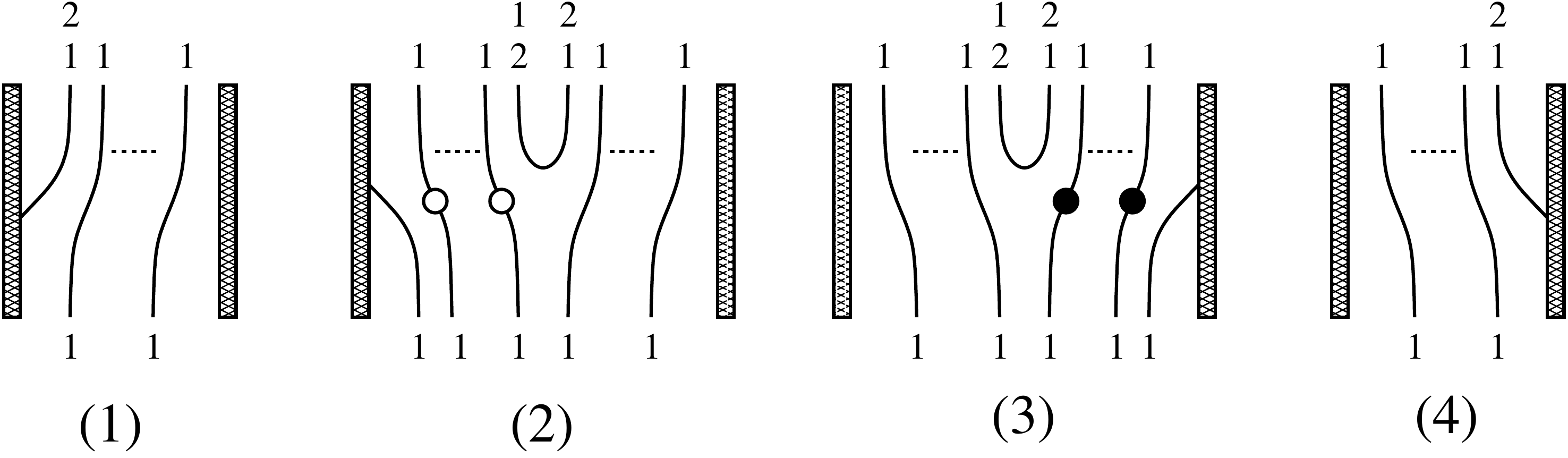}
\end{center}
\caption{Maps with image $e_1 \otimes \cdots \otimes e_1$}
\label{all1}
\end{figure}
 
\begin{corollary}
For every odd $n$, there exists a rank 2 submodule of $H_n(X)$ that is annihilated by 
multiplication by $y^4-1$.

For every even $n$, there exists a rank 1 submodule $K_1$ of $Z_n(X)$ that is in the boundary group $B_n(X)$,
and a rank 1 submodule $K_2$ that is annihilated by 
multiplication by $y^2-1$.
\end{corollary}

\begin{proof}
Let $n$ be odd.
Let $K$ be the rank 2 submodule of $C_n(X)$ generated by 
$e_{j} \otimes \cdots  \otimes  e_{j}$, $j=1,2$.
By Lemma~\ref{lem:aj}, $K$ is in $Z_n(X)$.
Since  $n+1$ is  even, Proposition~\ref{lem:bj} implies that 
 ${\rm Im}(d_{n+1})$ in  the submodule generated by $e_{1,0}$ in $Z_n(X)$ is spanned by 
${\rm GCD}\{ (y^{4(i-1)} -1) : i=2, \ldots, (n+1)/2 \} e_{1,0} $, and ${\rm GCD}\{ (y^{4(i-1)} -1) : i=2, \ldots, (n+1)/2 \} = y^4-1$. 
 Similarly,  ${\rm Im}(d_{n+1})$ in  the submodule generated by $e_{2,0}$ in $Z_n(X)$ is spanned by 
 ${\rm GCD}\{ (y^{2(n-2i-2)} -1) : i=2, \ldots, (n+1)/2 \} e_{2,0} $, and ${\rm GCD}\{ (y^{2(n-2i-2)} -1 ) : i=2, \ldots, (n+1)/2 \} = y^4-1$. Hence the rank 2 submodule of $Z_n(X)$ generated by $e_{j,0}$ for $j=1,2$ is annihilated by $(y^4-1)$. 
 
Let $n$ be even. Then $n+1$ is odd.
Let $K_j$ be the rank 1 submodule of $Z_n(X)$ generated by $e_{j,0}$ for $j=1,2$, respectively.
Since $ d_{n+1}(e_{1,1}) = ( y^2) e_{1,0} $ from Proposition~\ref{lem:bj} and $y^2$ is a unit, $K_1$ is in  ${\rm Im}(d_{n+1})$. 
The submodule $K_2$ is annihilated by the GCD of $y^{2(n+1)-4i}-1$ for $i=1, \ldots, n/2$, which is $y^2-1$. Thus the statement follows.
\end{proof}

The statement of the preceding corollary supports Przytycki-Wang's conjecture.

\bigskip

\noindent
{\bf Acknowledgement.}
We are grateful to Jozef Przytycki and Xiao Wang  for valuable conversations. 
MS was supported in part by NSF  DMS 1800443.

\bigskip

\appendix

\section{Proof of Theorem~\ref{thm:diff}}\label{apdx:diff}

In this section we use $\mu$ and $\lambda$ instead of $\mu_\ell$ and $\lambda_\ell$, respectively, for brevity. 
We need a few preliminary maps 
and results before proving the main theorem. 
First we introduce a class of operators $V^{\otimes n} \longrightarrow V^{\otimes n+1}$ whose diagrammatic interpretation is similar to the curtain differentials. Namely we set
\begin{eqnarray*}
\lefteqn{\Psi_n}\ \  &=& (\mu\otimes \mathbbm{1}\otimes \cdots \otimes  \mathbbm{1} )
\circ (R\otimes\mathbbm{1}\otimes \cdots \otimes \mathbbm{1})  \\
& & \circ \cdots \circ(\mathbbm{1}\otimes\cdots \mathbbm{1}\otimes R\otimes \mathbbm{1})
\circ (\mathbbm{1}\otimes\cdots \mathbbm{1}\otimes\beta),
\end{eqnarray*}
where $\mu$ indicates the action $\mu :k\otimes V \longrightarrow k$. We set $\Psi_0 = \lambda$, where $\lambda$ is the coaction $k\longrightarrow k\otimes V $. See Figure~\ref{Psi} for a diagram representing $\Psi_n$. We similarly define $\Psi'_n$, by symmetry, where we replace overpassing crossings with underpassing and the left action with the right action. 

\begin{figure}[htb]
\begin{center}
\includegraphics[width=.8in]{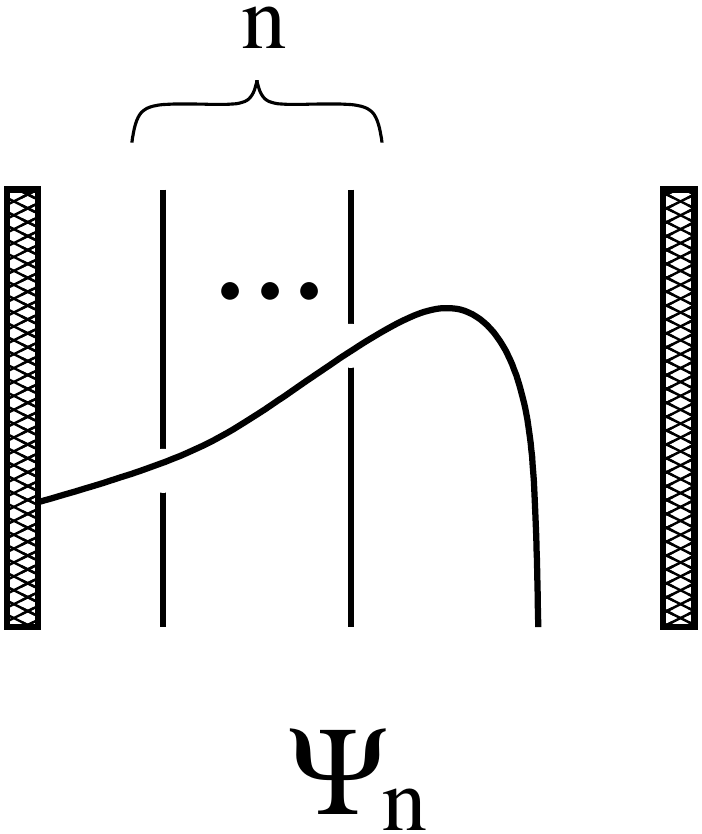}
\end{center}
\caption{Diagram representing $\Psi_n$}
\label{Psi}
\end{figure}

In the notation below, the dot $\cdot$ represents the horizontal concatenation of diagrams, that represents tensor product of maps. For example, if $f: V \rightarrow V$ is a map on $V$ and $\mathbbm{1}: V \rightarrow V$ denotes the identity map, then $f\cdot \mathbbm{1}$ denotes $f \otimes \mathbbm{1}$ on 
$V \otimes V$, and $\mathbbm{1}^k$ denotes the identity map on $V^{\otimes k}$. The dot may be abbreviated. 

\begin{remark}
	{\rm 
Lemmas~\ref{lem:leftdiff}, \ref{lem:psimap}, 
and \ref{lem:leftpsi} below are easily adapted, by symmetry, to the case of the right differentials upon exchanging $\Psi_n$, $\xi$, $\mu$ and $\lambda$ with $\Psi'_n$, $\zeta$, $\mu_r$ and $\lambda_r$, respectively.
}
\end{remark}

\begin{lemma}\label{lem:leftdiff}
	The left differentials $d^\ell_{n,n}$ can be decomposed in terms of $\Psi_k$ for all $n\in\N$ as follows:
$$
d_{n,n}^\ell = \sum_{m=2}^{n} \Psi_{n-m} \cdot \alpha \cdot \mathbbm{1}^{m-2}+\mu\cdot \mathbbm{1}^{n-1}.
$$
\end{lemma}

\begin{proof}
	The proof is by induction on $n$, the number of strings in the curtain representing $d_{n,n}^\ell$, i.e. the number of copies of $V$ in the domain of $d_{n,n}^\ell$. The base of the induction is easily verified by direct inspection. For $n=1$ the statement is in fact vacuously true, while for $n=2$ it is a consequence of the skein relation. Suppose the equation holds true for all $3\leq n \leq k$ and set $n = k+1$. Making use of the skein relation we can write 
	$$
	d_{k+1,k+1}^\ell = \Psi_{k-1} \cdot \alpha + d_{k,k}^\ell \cdot \mathbbm{1}.
	$$
	See Figure~\ref{Psiskein} for diagrams.
	Applying the inductive hypothesis to 
	$d_{k,k}^\ell$  we obtain 
	$$
	d_{k+1,k+1}^\ell = \Psi_{k-1} \cdot \alpha + \Psi_{k-2} \cdot \alpha\cdot \mathbbm{1} + \cdots  + \Psi_0 \cdot \alpha \mathbbm{1}^{k-2} \cdot \mathbbm{1} + \mu\cdot \mathbbm{1}^{k-1} \cdot \mathbbm{1},
	$$
	which concludes the proof. 
\end{proof}

\begin{figure}[htb]
\begin{center}
\includegraphics[width=3.2in]{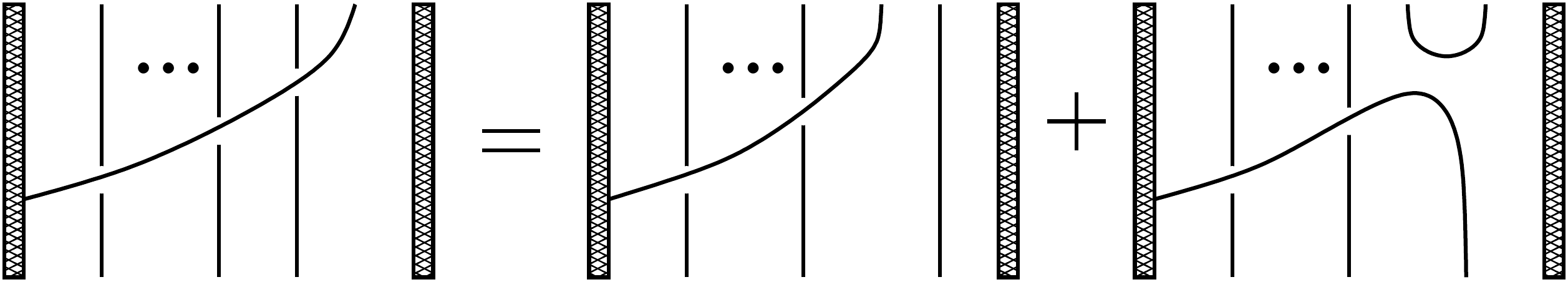}
\end{center}
\caption{Skein using $\Psi_n$}
\label{Psiskein}
\end{figure}

We now define 
the sets $\Gamma^{(n,n+2)}$ and $\Lambda^{(m,m)}$ of diagrams representing maps 
 that will be used to decompose the operators $\Psi_n$. In general the double superscripts
 ${M}^{(m,n)}$ indicates that the set includes maps $V^{\otimes m} \rightarrow V^{\otimes n}$.
We set (see Figure~\ref{GammaLambda}):
$$ \Gamma^{(0,2)} = \{ \beta \}, \quad   \Gamma^{(1,3)} = \{ \mathbbm{1} \beta, \beta \xi \} , \quad 
\Lambda^{(1,1)}=\{ \xi \}, \quad
  \Lambda^{(2,2)} = \{ \beta \alpha , \xi^2  \} 
$$
 and   inductively define 
\begin{eqnarray}
\quad	\Gamma^{(n,n+2)} &=& \Gamma^{(n-1,n+1)}\cdot \xi \cup \bigcup_{m=2}^{n-2} \Gamma^{(n-m,n-m+2)} \cdot \alpha\mathbbm{1}^{m-2}\beta\cup\{\mathbbm{1}^n\beta \}, \label{eqn1} \\
\quad	\Lambda^{(n,n)} &=& \Lambda^{(n-1,n-1)}\cdot \xi \cup \bigcup_{m=2}^{n-2} \Lambda^{(n-m,n-m)} \cdot \alpha\mathbbm{1}^{m-2}\beta\cup\{\alpha\mathbbm{1}^{n-2}\beta\}.  \label{eqn12}
	\end{eqnarray}
	
\begin{remark}
	{\rm 
It can be seen that the unions defining 
$\Gamma^{(n,n+2)}$ and $\Lambda^{(n,n)}$ are
in fact disjoint. 	
}
\end{remark}

\begin{figure}[htb]
\begin{center}
\includegraphics[width=2.5in]{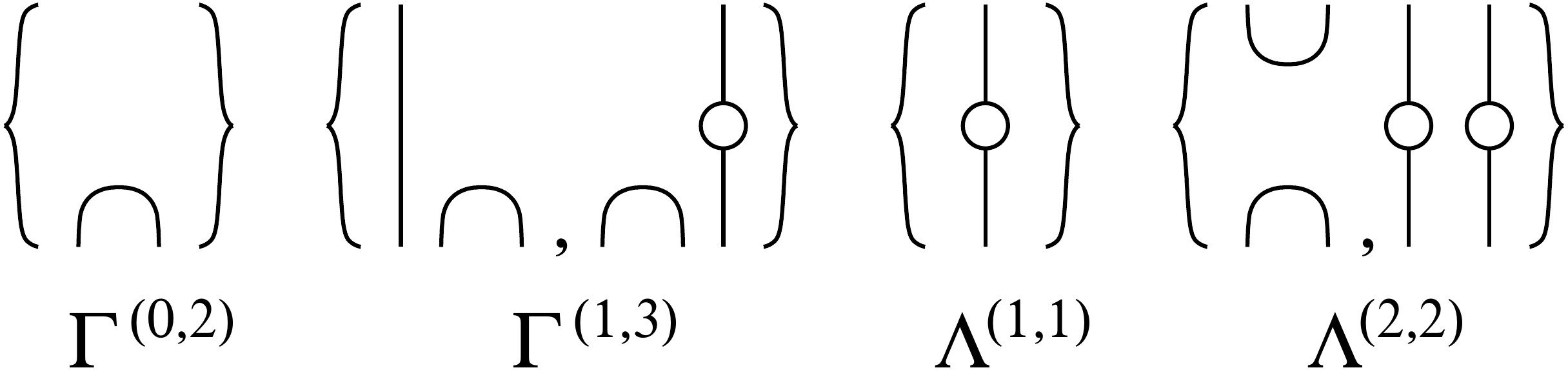}
\end{center}
\caption{$\Gamma$s and $\Lambda$s}
\label{GammaLambda}
\end{figure}

\begin{lemma}\label{lem:psimap}
	For all $n\in\N$ the following equation holds:
	$$
	\Psi_n = \sum_{\psi\in \Gamma^{(n-1,n+1)}} \mu\cdot \psi + \sum_{\phi\in \Lambda^{(n,n)}} \lambda\cdot \phi.
	$$
\end{lemma}

\begin{proof}
Recall that we use abbreviation $\mu_\ell=\mu$ and $\lambda_\ell = \lambda$. 
The proof utilizes induction and Lemma~\ref{lem:leftdiff}. A direct inspection shows that the equation holds for $n=1,2$. 
Indeed by using the skein relation we  have 
\begin{eqnarray*}
\Psi_1 &=& \mu \beta + \lambda \xi, \\
\Psi_2 &=&  \mu \mathbbm{1}\beta + \mu\beta\xi+ \lambda \alpha\beta + \lambda \xi \xi, 
\end{eqnarray*}
and it follows that the statement holds true for $n=1, 2$. Let us now assume that $\Psi_n$ is of the form given in the statement for all $2\leq n\leq k$ and let $n=k+1$. Applying the skein relation once to $\Psi_{k+1}$ we obtain
$$
\Psi_{k+1} = d_{k+1, k+1}^\ell\cdot \beta+ \Psi_k\cdot \xi 
$$ 
See Figure~\ref{Psiskein2} for the diagrams.

\begin{figure}[htb]
\begin{center}
\includegraphics[width=3.2in]{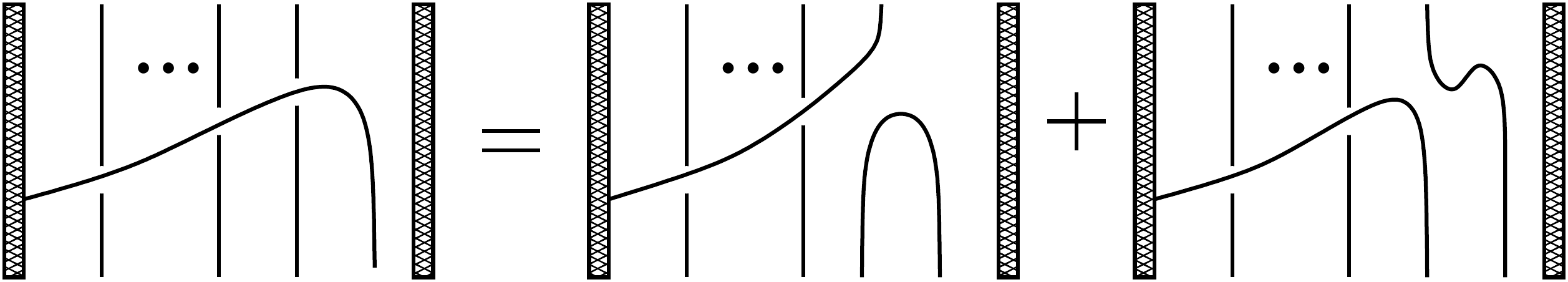}
\end{center}
\caption{Skein for $\Psi_n$}
\label{Psiskein2}
\end{figure}

Using Lemma~\ref{lem:leftdiff} we can rewrite the previous equation as 
$$
\Psi_{k+1} =  \Psi_k\cdot \xi + \Psi_{k-1}\cdot \alpha\beta + \Psi_{k-2}\cdot \alpha\mathbbm{1}\beta +\cdots +\Psi_0\cdot \alpha\mathbbm{1}^{k-1}\beta + \lambda\cdot \mathbbm{1}^k\beta. 
$$
We now apply the inductive hypothesis to obtain
\begin{eqnarray*}
	\lefteqn{\Psi_k}\ \ &=& \sum_{\psi\in \Gamma^{(k-1,k+1)}} \mu\cdot \psi \cdot  \xi+ \sum_{\phi\in \Lambda^{(k,k)}} \lambda\cdot \phi\cdot  \xi \\
	&& \vdots\\
	&& +  \sum_{\psi\in \Gamma^{(k-m-1,k-m+1)}} \mu\cdot \psi \cdot \alpha\mathbbm{1}^{m-2}\beta+ \sum_{\phi\in \Lambda^{(k-m,k-m)}} \lambda\cdot \phi\cdot \alpha\mathbbm{1}^{m-2}\beta\\
	&& \vdots\\
	&&+ \lambda\cdot \alpha\mathbbm{1}^{k-1}\beta + \mu\cdot \mathbbm{1}^k\beta.
	\end{eqnarray*}
Using the inductive definition of $ \Gamma^{(n,n+2)}$ and $ \Lambda^{(m,m)}$ we conclude that 
$$
\Psi_{k+1} = \sum_{\psi\in \Gamma^{(k,k+2)}} \mu\cdot \psi + \sum_{\phi\in \Lambda^{(k+1,k+1)}} \lambda\cdot \phi,
$$
which concludes the proof of the lemma.
\end{proof}

\begin{lemma}\label{lem:leftpsi}
	The left differential $d_n^\ell$ can be written in terms of $\Psi_i$ 
	as follows:
	$$
	d_n^\ell = \begin{cases}\sum_{i=1}^{k} \Psi_{2(i-1)} \cdot \alpha \cdot \mathbbm{1}^{n-2i}\ \quad {\rm for} \quad  n = 2k\\[8pt]
	- 
	\mu \mathbbm{1}^{n-1} - \sum_{j=1}^{k} \Psi_{2j-1}\cdot \alpha \cdot \mathbbm{1}^{n-2j+1}\ \quad {\rm for} \quad n = 2k+1.
	\end{cases}
	$$
\end{lemma}
\begin{proof}
	Suppose first that $n$ is even and let $n = 2k$ for some $k$. By definition we have
	$$
	d^\ell_n = \sum_{i=1}^n (-1)^i d_{i,i}^\ell \cdot \mathbbm{1}^{n-i}.
	$$
	Since $n$ is even, we can group the terms $d_{i,i}^\ell\cdot \mathbbm{1}^{n-i}$ in pairs of consecutive summands $2i$ and $2i+1$ for $i=0,\ldots,k$. Applying the skein relation
	to the 
	$(2i)^{\rm th}$ 
	term, 
we obtain that, for all $i$, 
	 $d_{2i, 2i}^\ell = d_{2i-1, 2i-1}^\ell + \Psi_{2i-2}\cdot \alpha$. Here we recall that $\Psi_0=\lambda$.   
	Putting all terms  of the left-hand side of $- d_{2i-1, 2i-1}^\ell  + d_{2i, 2i}^\ell = \Psi_{2i-2}\cdot \alpha$ together and using the fact that consecutive terms appear with opposite signs, 
we complete the proof for the case $n$ even. If $n = 2k +1$ is odd, we proceed similarly by grouping in pairs the terms $d_{2j,2j}^\ell$ and $d_{2j+1,2j+1}^\ell$ for $j=1,\ldots, k$. 
\end{proof}

\noindent
{\it Proof of Theorem~\ref{thm:diff}}. 
We first consider the case $n = 2s$ for some $s$. 
Since inn Lemma~\ref{lem:triviald2} below we show that $d_2^\ell = 0$, we assume that $s\geq 2$. 
From Lemma~\ref{lem:leftpsi} we have
$$
d^\ell_n = \sum_{i=1}^{s} \Psi_{2(i-1)} \cdot \alpha \cdot \mathbbm{1}^{n-2i}. 
$$
Using Lemma~\ref{lem:psimap} we can rewrite it as
$$
d^\ell_n = \sum_{i=2}^s \left(\sum_{\psi\in \Gamma^{(2i-3,2i-1)}} \mu\cdot \psi + \sum_{\phi\in \Lambda^{(2i-2,2i-2)}} \lambda\cdot \phi\right) \cdot \alpha \mathbbm{1}^{n-2i}.
$$ 
To complete the proof of the first assertion with even $n$, it would suffice to show that for each $i = 2, \ldots , s$ 
$$
\sum_{\psi\in \Gamma^{(2i-3,2i-1)}} \mu\cdot \psi \cdot \alpha + \sum_{\phi\in \Lambda^{(2i-2,2i-2)}} \lambda\cdot \phi\cdot \alpha = \sum_{S'(n)} g'_{i_0} g_{i_1}^{k(1)} \cdots g_{i_h}^{k(h)},
$$
where, noting $2k=n=2i$, 
 $$S'(n) =\{ (i_0,i_1,\ldots , i_h;k(1),\ldots ,  k(h))\ |\ i_h\neq 1,\ i_0+ i_1^{k(1)}+\cdots + i_h^{k(h)} = 2i\}.$$
 Since $n-2i = 2s -2i$ is even for all $i = 2,\ldots , s$, 
 $g'_1 = \mu$ by definition and $g'_{i_0}$ contains a factor of $\lambda$ for each $i_0 \geq 2$, the last equality is a consequence of the two set-theoretic equalities:
 \begin{eqnarray*}
 \lefteqn{ 	 \Gamma^{(\ell,\ell+2)}\cdot \alpha  }\\
 & =\ &  \{g_{i_1}^{k(1)}\cdots g_{i_h}^{k(h)}\ |\ i_h\neq 1,\ i_1^{k(1)}+\cdots + i_h^{k(h)} = \ell+2 \}  \\
 &=: & S'_1(\ell),\\
  \lefteqn{ 	 	\lambda\cdot \Lambda^{(\ell,\ell)}\cdot \alpha } \\
&=\ &    \{g'_{i_0}\cdot g_{i_1}^{k(1)}\cdots g_{i_h}^{k(h)}\ |\ i_h\neq 1,\ i_0\neq 1,\ i_0+ i_1^{k(1)}+\cdots + i_h^{k(h)} = \ell + 2 \} \\ & =: & S'_2(\ell),
 	\end{eqnarray*}
 for all $n$, where $g_{i_1}^{k(1)}\cdots g_{i_h}^{k(h)}$ can be empty in the second last line. 
   We therefore proceed to prove the first equality by induction. First observe that by definition, for $\ell=1$, we have 
 $$ 
 \Gamma^{(1,3)} = \{\mathbbm{1}\beta,\beta\xi \},
 $$
 from which 
 $$
   \Gamma^{(1,3)}\cdot \alpha = \{\mathbbm{1}\beta\alpha,\beta\xi\alpha \}.
 $$
 It is easy to see by direct inspection that 
 $$
 \{g_{i_1}^{k(1)}\cdots g_{i_h}^{k(h)}\ |\ i_h\neq 1,\ i_1^{k(1)}+\cdots +i_h^{k(h)} = 3 \} = \{g_1\cdot g_2, g_3\} = \{\mathbbm{1}\beta\alpha,\beta\xi \alpha \}.
 $$
 So the basis of induction holds true. Let us now assume the the equality holds for all $\ell$ smaller than or equal to $r$, and suppose $\ell = r+1$. We want to show the inclusion 
 $$
  \Gamma^{(r+1,r+3)}\cdot \alpha \subset S'_1(r+1).
 $$
 Let $\psi\in  \Gamma^{(r+1,r+3)}$. From the Equality~\ref{eqn1}, we distinguish three cases
 $$
 \psi \in \begin{cases}
  \Gamma^{(r,r+2)} \cdot \xi \\[5pt]
 \bigsqcup_{m=2}^{r-1} \Gamma^{(r+1-m,r+3-m)}\cdot \alpha\mathbbm{1}^{m-2}\beta\\[5pt]
 \{\mathbbm{1}^{r+1}\beta \}.
 \end{cases}
 $$
 In the last case it is clear that $\psi\alpha = \mathbbm{1}^{r+1}\beta\alpha \in S'_1(r+1)$. In the second case, $\psi$ is equal to $\psi'\cdot \alpha\mathbbm{1}^{m-2}\beta$, for $\psi'\in \Gamma^{(r+1-m,r+3-m)}$ for some $m = 2,\ldots,  r-1$. 
 Then 
 $\psi\cdot \alpha = \psi' \cdot \alpha \cdot \mathbbm{1}^{m-2}\beta\alpha \in S'_1(r+1)$
  since $\mathbbm{1}^{m-2}\beta\alpha\in S'_1(m)$, $\psi' \cdot \alpha\in S'_1(r-m+1)$ by inductive hypothesis and $S'_1(n)\cdot S'_1(m) \subset S'_1(n+m)$ for all $n,m$. 
  
  Lastly, if $\psi\in \Gamma^{(r,r+2)}\cdot   \xi $ we can write $\psi\cdot \alpha = p_1\cdot   \xi \cdot \alpha$ for some $p_1\in  \Gamma^{(r,r+2)}$. We again distinguish three subcases depending on which 
  of the three cases $p_1$ belongs to. 
     As before, we see that if $p_1$ is not in $ \Gamma^{(r-1,r+1)}\cdot   \xi $ we easily have that $\psi\cdot \alpha\in S'_1(r+1)$. Otherwise we can write $\psi\cdot \alpha = p_2\cdot   \xi   \xi  \cdot \alpha$. So proceeding, at each step we have that either $\psi\cdot \alpha \in S'_1(r+1)$, or we decompose $\psi\cdot \alpha$ as a product of type $p_k\cdot   \xi ^{k}\alpha$, with $p_k\in \Gamma^{(1,3)} = \{\mathbbm{1}\beta,\beta  \xi  \}$. Either way $\psi\cdot \alpha \in S'_1(r+1)$ and we have proved that 
 $$
  \Gamma^{(r+1,r+3)}\cdot \alpha \subset S'_1(r+1).
 $$
 We now show the opposite inclusion. Let $g = g_{i_1}^{k(1)}\cdots g_{i_h}^{k(h)} \in S'_1(r+1)$ with $i_h \neq 1$. 
 We can therefore write $g = g_{i_1}^{k(1)}\cdots g_{i_{h-1}}^{k(h-1)}\beta  \xi ^{i_h-2}\alpha \cdots \beta  \xi ^{i_h-2}\alpha$, 
 where $i_h-2$ 
 can be possibly zero, and $\beta  \xi ^{i_h-2} \alpha$ appears $k(h)$ times. 
 We also abbreviate center dots for brevity, such as $\beta \alpha$ for $\beta \cdot \alpha$, with the understanding that these sequences denote the horizontal concatinations instead of compositions of maps.
 If $i_j =1$ for all $j = 1,\ldots , h-1$, we have that 
 $$g = \mathbbm{1}^{k(1)+\cdots + k(h-1)}\cdot \beta\cdot   \xi ^{i_h-2}\beta\alpha  \xi ^{i_h-2}\beta\alpha\cdots \beta\alpha  \xi ^{i_h-2}\cdot\alpha\in  \Gamma^{(r+1,r+3)}\cdot \alpha, $$
 since it is easily seen that 
 $$\mathbbm{1}^{k(1)+\cdots + k(h-1)}\cdot \beta\cdot   \xi ^{i_h-2}\beta\alpha  \xi ^{i_h-2}\beta\alpha\cdots \beta\alpha  \xi ^{i_h-2}\in  \Gamma^{(r+1,r+3)}. $$
  Otherwise let $j$ be the largest index for which $i_j \neq 1$. Then we have 
  $$g = g_{i_1}^{k(1)}\cdots g_{i_{j-1}}^{k(j-1)}\cdot \beta  \xi ^{i_j-2}\alpha \cdots \beta  \xi ^{i_j-2}\alpha \cdot \mathbbm{1}^{i_{j+1}+\cdots+i_{h-1}} \cdot \beta  \xi ^{i_h-2}\alpha. $$
  By the induction hypothesis we can write 
 $$
 g = p\cdot \alpha \cdot \beta  \xi ^{i_j-2}\alpha\mathbbm{1}^{i_{j+1}+\cdots+i_{h-1}}\cdot \beta  \xi ^{i_h-2}\alpha,
 $$
 for $p\in  \Gamma^{(m,m+2)}$ for some $m\leq r$. Since 
 $$p\cdot \alpha \cdot \beta  \xi ^{i_j-2}\alpha\mathbbm{1}^{i_{j+1}+\cdots+i_{h-1}}\cdot \beta  \xi ^{i_h-2}\in  \Gamma^{(r+1,r+3)}, $$
 we conclude that $g \in  \Gamma^{(r+1,r+3)}\cdot \alpha$. Therefore $ \Gamma^{(\ell,\ell+2)}\cdot \alpha = S'_1(\ell)$ for all even $n$. To prove that $ \Lambda^{(\ell,\ell)}\cdot \alpha = S'_2(\ell)$, we again proceed by induction. The proof is similar to the case of $\Gamma^{(n,n+2)}$. The base of induction is holds true since we have
 $$
 S'_2(1) = \{\lambda \xi\alpha \},
 $$
 and 
 $$
 \Lambda^{(1,1)} = \{\xi \}.
 $$
 Let us know suppose that the equality $ \Lambda^{(\ell,\ell)}\cdot \alpha = S'_2(\ell)$ holds for all $2\leq \ell\leq r$. We want to show $ \Lambda^{(r+1,r+1)}\cdot \alpha = S'_2(r+1)$. Consider again three different cases
 $$
 \phi\in \begin{cases}
 \Lambda^{(r,r)}\cdot \xi\\
\bigsqcup_{m=2}^r \Lambda^{(r-m,r-m)}\cdot \alpha\mathbbm{1}^{m-2}\beta\\
 \{\alpha\mathbbm{1}^{r-1}\beta \}.
 \end{cases}
 $$
 In the first case, $\phi = q\cdot \xi$ for some $q\in \Lambda^{r,r}$ and we can proceed backward as for the analogous proof for $\Gamma^{(r+1,r+3)}$ so that at each step we either have $\phi\in S'_2(r+1)$ or we can rewrite $\phi = \tilde{q}\cdot \xi \cdots \xi$, where the product of $\xi$ is $r$ times and $\tilde{q} \in \Lambda^{1,1}$. It follows that in the first case $\lambda \phi \alpha\in S'_2(r+1)$. In the second case we have $\phi = q\cdot \alpha\mathbbm{1}^{r-m}\beta$, for some $q \in \Lambda^{(r-m,r-m)}$. Since $q\cdot \alpha\in S'_2(t)$ for some $t$ by induction, and $\mathbbm{1}^{r-m}\alpha\beta$ is of type $g_1^d\cdot g_2$, this case follows as well.  In the third case, $\lambda \alpha\mathbbm{1}^{r-1}\beta \cdot \alpha = \lambda \alpha \mathbbm{1}^{r-1}\alpha\beta \in S'_2(r+1)$.  It follows that $d_n^\ell$ decomposes as in the statement of the theorem, when $n$ is even. The case $n$ odd is similar. Let $n = 2s+1$ for some $s$, then using Lemma~\ref{lem:leftpsi}, odd case, it holds
 $$
 d^\ell_{2s+1} = -\mu \mathbbm{1}^{2s} - \sum_{j=1}^2 \Psi_{2j-1}\alpha \mathbbm{1}^{2s-2j}.
 $$
 Applying Lemma~\ref{lem:psimap} we obtain
 $$
 d^\ell_{2s+1} = -\mu \mathbbm{1}^{2s} - \sum_{j=1}^s  \left( \sum_{\psi\in\Gamma^{(2j-2,2j)}} \mu\cdot \psi + \sum_{\phi\in\Lambda^{(2j-1,2j-1)}}\lambda\cdot \phi \right)\cdot \alpha\mathbbm{1}^{2s-2j}.
 $$
 Since $\Gamma^{(2j-2,2j)}$ has even exponents for all $j$'s, using the recursive definition of $\Gamma$'s it follows that there is no term of type $\mu\cdot \mathbbm{1}^d$ in the sum $\sum_{\psi\in\Gamma^{(2j-2,2j)}} \mu\cdot \psi$. So it is enough to show that for each $j = 1,\ldots , s$ we have
 $$
 \left( \sum_{\psi\in\Gamma^{(2j-2,2j)}} \mu\cdot \psi + \sum_{\phi\in\Lambda^{(2j-1,2j-1)}}\lambda\cdot \phi \right)\cdot \alpha = \sum_{S''_2(n)} g'_{i_0}g_{i_1}^{k(1)}\cdots g_{i_h}^{k(h)},
 $$
 where the sum runs over all tuples in 
 $$
 S''_2(n) := \{(i_0,i_1,\ldots, i_h;k(1) , \ldots,k(h) \mid i_0 + i_1^{k(h)} +\cdots +  i_h^{k(h)} = 2j+1 \}. 
 $$
 Since $n-2j = 2s+1-2j$ is odd for all $j$ and $g'_{i_0}$ contains a factor of $\lambda$ for all $i_0 \geq 2$ it follows that it is enough to prove the set theoretic equalities 
 \begin{eqnarray*}
 \Gamma^{(d,d+2)} &=&S'_1(d)\\
 \lambda\cdot \Lambda^{(d,d)} \alpha &=& S'_2(d).
 \end{eqnarray*}
These have already been proved above and the proof for $n = 2s+1$ is complete as well. 

 It is easy to see that mirroring Lemmas~\ref{lem:leftdiff},~\ref{lem:psimap},~\ref{lem:leftpsi} with respect to the $y$-axis, we obtain a decomposition of $ d_n^r$ with right (co)action on ${k}$ replacing the left (co)action on ${k}$, and $\zeta$ 
 instead of $  \xi $. So the formula for $ d_n^\ell$ just proved can be easily adapted for $ d_n^r$. Putting the two equations together and distinguishing the cases $n$ odd and even, we conclude the proof of the theorem. \qed 

\section{Proof of Proposition~\ref{lem:bj} continued}\label{apdx:higher}

\flushleft 

In this section, we provide proofs of the other cases. 
For odd $n$, the following terms result in the non-zero coefficient of $e_{1,0}$ in the image, for $i=1, \ldots, (n-1)/2$:
\begin{eqnarray*}
( g_1' g_1^{n-1} ) ( e_{1, 0} ) &=& 1   e_{1, 0} \\
( g_1' g_1^{n-1} ) (  e_{1,1}) &=& 1  e_{1, 0} \\
(g_{2i+1}'  g_1^{n-2i - 1 } ) ( e_{1, 2i}  )&=& (-y) (y^{-1} ) y^{2( 2i-1) }  e_{1, 0}  \quad = \quad  - y^{ 4i-2 }  e_{1, 0} \\
(g_{2i+1}'  g_1^{n-2i - 1 } ) (  e_{1, 2i+1} ) &=& (-y) ( -y ) y^{2( 2i-1) } e_{1, 0}  \quad = \quad  y^{4i}   e_{1, 0}\\
(h_1^{2i-2} h_{n-2i+2}' ) ( e_{1, 2i - 1 } ) &=&  ( y^{-1} )y  e_{1, 0} \quad = \quad 1  e_{1, 0}  \\
 (h_1^{2i-2 } h_{n-2i +2  }' ) ( e_{1, 2i})  &=&  (-y  )y e_{1, 0}  \quad = \quad -y^2  e_{1, 0}\\
(h_1^{n-1} h_1' ) ( e_{1, 0} )&=& 1  e_{1, 0}\\
(h_1^{n-1} h_1' ) ( e_{1,n} ) &=& 1  e_{1, 0} .
\end{eqnarray*} 

From these we compute
\begin{eqnarray*}
d_n(e_{1,0}) &=& ( -  g_1' g_1^{n-1}+ h_1^{n-1} h_1'  ) ( e_{1, 0} ) \quad =\quad (-1+1 ) e_{1, 0}\quad =\quad 0 \\
d_n(e_{1,1}) &=& ( - g_1' g_1^{n-1} - g_{3}'  g_1^{n-3} + h_{n}' ) ( e_{1, 1} ) \quad =\quad  (-1+y^2  +1 ) e_{1, 0}\quad =\quad  y^2  e_{1, 0}\\
d_n(e_{1,2i-1}) &=& ( - g_{2i-1}'  g_1^{n-2i + 1 } +h_1^{2i-2} h_{n-2i+2}'  ) (e_{1,2i-1}) 
 \quad =\quad  (- y^{4i-4}+ 1 )  e_{1, 0}\\
d_n(e_{1,2i}) &=&  (- g_{2i+1}'  g_1^{n-2i -1 } +  h_1^{2i-2} h_{n-2i-2}' ) ( e_{1, 2i}) 
 \quad =\quad (  y^{4i-2} - y^2 ) e_{1, 0} \\
d_n(e_{1,n}) &=&  (- g_{n}'  +  h_1^{n-1} h_{1}' ) ( e_{1, n}) 
 \quad =\quad  (- y^{2(n-1)} + 1 ) e_{1, 0}.
\end{eqnarray*}

For the coefficient of $e_{2,0}$, we compute 
\begin{eqnarray*}
( g_1' g_1^{n-1} ) ( e_{2, 0} ) &=& 1 e_{2,0} \\
( g_1' g_1^{n-1} ) ( e_{2,1}) &=& 1 e_{2,0}\\
(g_{2i+1}'  g_1^{n-2i - 1 } ) ( e_{2, 2i} ) 
&=& (y) ( -y) e_{2,0}\quad = \quad - y^2 e_{2,0}\\
(g_{2i+1}'  g_1^{n-2i - 1 } ) ( e_{2, 2i+1} ) 
& = &  (y) ( y^{-1}  ) e_{2,0}\quad = \quad 1 e_{2,0}\\
(h_1^{2i-2} h_{n-2i+2}' ) ( e_{2, 2i-1 } ) 
&=&  ( -y  )y \cdot y^{2(n-2i)}  e_{2,0} \quad = \quad   - y^{2n-4i+2}e_{2,0}  \\
(h_1^{2i-2} h_{n-2i+2}' ) ( e_{2, 2i} ) 
&=&  (y^{-1})y \cdot y^{2(n-2i)}  e_{2,0} \quad = \quad    y^{2n-4i }  e_{2,0} \\
(h_1^{n-1} h_1' ) ( e_{2, 0}  ) &=& 1 e_{2,0} \\
(h_1^{n-1} h_1' ) (  e_{2,n} ) ) &=& 1 e_{2,0} .
\end{eqnarray*}

From these we compute
\begin{eqnarray*}
d_n(e_{2,0}) &=& ( -  g_1' g_1^{n-1}+ h_1^{n-1} h_1'  ) ( e_{2, 0} ) \quad =\quad (-1+1) e_{2,0} \quad =\quad 0 \\
d_n(e_{2,1}) &=& ( - g_1' g_1^{n-1} + h_3' h_1^{n-3} ) 
 ( e_{2, 1} ) \quad =\quad (-1+y^{2n-2} )e_{2,0} \\
d_n(e_{2,2i+1}) &=& ( - g_{2i+1}'  g_1^{n-2i -1 } +h_1^{2i} h_{n-2i-1}'  ) (e_{2,2i+1}) 
 \quad =\quad (  y^2   - y^{2n-4i-2}  ) e_{2,0} \\
d_n(e_{2,2i}) &=&  (- g_{2i+1}'  g_1^{n-2i -1 } +  h_1^{2i-2} h_{n-2i+2}' ) ( e_{2, 2i}) 
 \quad =\quad  ( -1 +   y^{2n-4i} ) e_{2,0}  \\
d_n(e_{2,n}) &=&  (- g_{n}'  +  h_1^{n-1} h_{1}' ) ( e_{2, n}) 
 \quad =\quad  ( - 1 + 1 )e_{2,0} \quad =\quad 0   .
\end{eqnarray*}

For even $n$ and for $e_{2,0}$, we have
\begin{eqnarray*}
(g_{2i}'  g_1^{n-2i} ) ( e_{2, 2i-1} ) &=& (-y) (y^{-1}  ) e_{2,0}
  \quad = \quad   -1    e_{2,0} \\
(g_{2i}'  g_1^{n-2i} ) ( e_{2,2i})  &=& (-y) (-y  ) e_{2,0}
  \quad = \quad  y^2 e_{2,0} \\
( h_1^{2i} h_{n-2i}' ) ( e_{1, 2i+1} )
&=& y ( y^{-1}) y^{2(n-2i)} e_{2,0}. \quad  = \quad   - y^{2(n-2i+1)}  e_{2,0}  \\
( h_1^{2i} h_{n-2i}' ) (  e_{1, 2i+2}) 
&=& y ( -y ) y^{2(n-2i+2)} e_{2,0}. \quad  = \quad    y^{2(n-2i)}  e_{2,0}  .
\end{eqnarray*}

From these we compute
\begin{eqnarray*}
d_n(e_{2,2i-1}) &=& ( -  g_{2i}' g_1^{n-2i}+ h_1^{2i-2} h_{n-2i+2}'  ) ( e_{1, 2i-1} ) \quad =\quad (y^2 -  y^{2(n-2i+3)} ) e_{2,0}\\
d_n(e_{2,2i}) &=& ( - g_{2i}' g_1^{n-2i} - h_{2i-2}'  h_{n-2i+2}' ) ( e_{1, 2i} ) \quad =\quad  ( - 1 + y^{2(n-2i+2)} ) e_{2,0}. \\
\end{eqnarray*} 
This completes the proof.


\begin{thebibliography}{0}
	\bib{CES}{article}{
		abstract = {A homology theory is developed for set-theoretic Yang-Baxter equations, and knot invariants are constructed by generalized colorings by biquandles and Yang-Baxter cocycles.},
		author = {J. Scott Carter},
		author={ Mohamed Elhamdadi},
		author={ Masahico Saito},
		journal = {Fundamenta Mathematicae},
		keywords = {knots; set-theoretic Yang-Baxter equation; homology; virtual links; biquandles},
		number = {1},
		pages = {31-54},
		title = {Homology theory for the set-theoretic Yang-Baxter equation and knot invariants from generalizations of quandles},
		url = {http://eudml.org/doc/282699},
		volume = {184},
		year = {2004},
	}
\bib{CJKLS}{article}{
	author={Carter, J. Scott},
	author={Jelsovsky, Daniel},
	author={Kamada, Seiichi},
	author={Langford, Laurel},
	author={Saito, Masahico},
	title={Quandle cohomology and state-sum invariants of knotted curves and
		surfaces},
	journal={Trans. Amer. Math. Soc.},
	volume={355},
	date={2003},
	number={10},
	pages={3947--3989},
	issn={0002-9947},
	review={\MR{1990571}},
}
\bib{Eis}{article}{
	title={Yang--Baxter deformations of quandles and racks},
	author={Eisermann, Michael},
	journal={Algebraic \& Geometric Topology},
	volume={5},
	number={2},
	pages={537--562},
	year={2005},
	publisher={Mathematical Sciences Publishers}
}
\bib{LeVe}{article}{
	title={Homology of left non-degenerate set-theoretic solutions to the Yang--Baxter equation},
	author={Lebed, Victoria},
	 author= {Vendramin, Leandro},
	journal={Advances in Mathematics},
	volume={304},
	pages={1219--1261},
	year={2017},
	publisher={Elsevier}
}
\bib{Jim}{article}{
	title={Introduction to the Yang-Baxter equation},
	author={Jimbo, Michio},
	journal={International Journal of Modern Physics A},
	volume={4},
	number={15},
	pages={3759--3777},
	year={1989},
	publisher={World Scientific}
}
	\bib{Prz}{book}{,
		author = { Józef H.   Przytycki },
		title = {Knots and Distributive Homology: From Arc Colorings to Yang–-Baxter Homology},
		publisher = {New Ideas in Low Dimensional Topology},
		pages = {413--488},
		URL = {https://www.worldscientific.com/doi/abs/10.1142/9789814630627_0011}
	}
	\bib{PW}{article}{
		author = {Przytycki, J{\'o}zef H.},
		author= {Wang, Xiao},
		title = {Equivalence of two definitions of set-theoretic Yang–Baxter homology and general Yang–Baxter homology},
		journal = {Journal of Knot Theory and Its Ramifications},
		volume = {27},
		number = {07},
		year = {2018},
		
		URL = { 
			https://doi.org/10.1142/S0218216518410134}
		}
	\bib{PW2}{article}{
		author = {Przytycki, J{\'o}zef H.},
		author= {Wang, Xiao},
		title = {Second Yang-Baxter homology for the HOMFLYPT polynomial},
		journal = {Preprint},
				year = {2020}
			}
		\bib{Tur}{article}{
			title={The Yang-Baxter equation and invariants of links.},
			author={Turaev, Vladimir G},
			journal={Inventiones mathematicae},
			volume={92},
			number={3},
			pages={527--554},
			year={1988}
		}
	\bib{Tur2}{book}{
		title={Quantum invariants of knots and 3-manifolds},
		author={Turaev, Vladimir G},
		volume={18},
		year={2016},
		publisher={Walter de Gruyter GmbH \& Co KG}
	}
	\bib{Wang}{article}{
		title={Knot theory and algebraic structures motivated by and applied to knots, PhD dissertation},
		author={Wang, Xiao},
		year={2018},
		journal={The George Washington University}
	}
\end{thebibliography}
\end{document}